\documentclass[reqno,12pt]{amsart} 
\usepackage{amscd,amssymb,pstricks,latexsym,amsbsy,xypic,mathrsfs,verbatim}
\usepackage[all]{xy}
\usepackage{lscape}
\usepackage{enumerate}
\usepackage{graphicx}

\makeatletter
\@namedef{subjclassname@2010}{%
  \textup{2010} Mathematics Subject Classification}

\numberwithin{equation}{section}

\theoremstyle{plain}
\newtheorem{prop}{Proposition}[section]
\newtheorem{thm}[prop]{Theorem}
\newtheorem{cor}[prop]{Corollary}
\newtheorem{lem}[prop]{Lemma}
\newtheorem{defn}[prop]{Definition}

\theoremstyle{definition}
\newtheorem{example}[prop]{Example}
\newtheorem{rem}[prop]{Remark}

\def\bbn{{\mathbb N}}
 \def\lz{\lambda}

\def\dim{{{\rm dim}\,}}

\def\End{{{\rm End}\,}} 
\def\Hom{{{\rm Hom}\,}}

\def\Ext{{{\rm Ext}\,}}

\def\mod{{\text{\rm mod}}}

\def\cal{\mathcal}

\def\dim{{\rm dim}\,}

\def\bbZ{{\mathbb Z}}
\def\bbp{{\mathbb P}}

\def\bbL{{\mathbb L}}

\def\Stab{{\rm Stab}}

\def\co{{\mathcal O}}
\def\vx{{\vec{x}}}
\def\vy{{\vec{y}}}
\def\vz{{\vec{z}}}
\def\vc{{\vec{c}}}

\def\ev{\mbox{\rm ev}}
\def\coh{{\rm coh\,}}
\def\vect{{\rm vect\,}}
\def\X{{\mathbb X}}

\def\coker{{\rm coker\,}}

\begin{document}

\title[{t-stabilities for a weighted projective line}]{t-stabilities for a weighted projective line}

\author{Shiquan Ruan and Xintian Wang}
\address{Xiamen University, Xiamen 361005, China}
\email{shiquanruan@stu.xmu.edu.cn }
\address{AMSS, Chinese Academy of Sciences, Beijing 100190, China}
\email{wangxintian0916@126.com }



\date{\today}


\begin{abstract}
The present paper focuses on the study of t-stabilities on a
triangulated category in the sense of
Gorodentsev, Kuleshov and Rudakov. We give an equivalent description
for the finest t-stability on a piecewise hereditary triangulated
category and, describe the semistable
subcategories and final HN triangles for (exceptional) coherent
sheaves in $D^b(\coh\X)$, which is the bounded derived category
of coherent sheaves on the weighted projective line $\X$ of weight type (2). Furthermore, we show the existence of a t-exceptional triple for $D^b(\coh\X)$. As an
application, we obtain a result of Dimitrov--Katzarkov which states that each stability condition $\sigma$ in the sense of Bridgeland admits a $\sigma$-exceptional triple for the
acyclic triangular quiver $Q$. Note that this implies the connectedness
of the space of stability conditions associated to $Q$.
\end{abstract}

\maketitle
\section{Introduction}

The notion of a stability condition on a triangulated category was
first introduced by Bridgeland in \cite{TB}. The motivation comes
from the study of Dirichlet branes in string theory in physics, and
especially from Douglas's work on $\Pi$-stability. The main result
in \cite{TB} states that the space $\rm{Stab}(\mathcal{C}$) of all
locally finite stability conditions on an essentially small
triangulated category $\cal C$ is a complex manifold, with a natural
right action by the group $\widetilde{\rm{GL}}^+(2,\mathbb{R})$ of
the universal covering space of the group of rank two matrices with
positive determinant, and a left action by the group
$\rm{Aut}(\mathcal{C})$ of exact autoequivalences of $\mathcal{C}$.
This space carries an interesting geometric and topological
structure which reflects the properties of $\mathcal{C}$. Moreover,
stability conditions are related to many mathematical subjects, such
as Donaldson-Thomas (DT) theory, homological mirror symmetry theory
and so on.

Obviously, one would like to be able to compute the spaces
$\rm{Stab}(\mathcal{C}$) of stability conditions in some interesting
examples, such as the bounded derived category $\mathcal{C}=D^b(\coh
\X)$ of coherent sheaves over smooth projective varieties $\X$, or
the bounded derived category $\mathcal{C}=D^b(\mod kQ)$ of finite
dimensional modules over the path algebra $kQ$ for a quiver $Q$. For
$\mathcal{C}=D^b(\coh \X)$, Bridgeland \cite{TB} first dealt with
the elliptic curve case; and Macri \cite{EM} considered any smooth projective
curve over $\mathbb{C}$ of positive genus, showing that the action
of $\widetilde{\rm{GL}}^+(2,\mathbb{R})$ on the subspace of locally
finite numerical stability conditions is free and transitive, moreover, he also investigated the case of projective line
$\mathbb{P}^1$ (by using exceptional sequences), which was first dealt by Okada \cite{SO}. For
$\mathcal{C}=D^b(\mod kQ)$, many people have worked for $Q$ of
Dynkin type, c.f. \cite{BQS,Qiu,S}; for $n$-Kronecker quiver $Q$
with $n\geq 2$, Macri proved in \cite{EM} that
$\rm{Stab}(\mathcal{C}$) is a connected and simply connected
2-dimensional complex manifold, and Dimitrov--Karzarkov described this manifold in more details in \cite{DK3}. Recently, Dimitrov and
Katzarkov \cite{DK, DK2} worked on the acyclic triangular quiver
$Q$, showing that $\Stab(\mathcal{C})$ is connected and
contractible.

In \cite{EM},  Macri gave a procedure generating stability
conditions from exceptional sequences. More precisely, he gave a
natural way to associate to a complete Ext-exceptional sequence a
heart of a bounded t-structure and then a family of stability
conditions which have this one as heart. In this way, one can define
a collection of open connected subsets of $\rm{Stab}(\mathcal{C}$) of
maximal dimension, parametrized by the orbits of the action of the
braid group on exceptional sequences. This method provides a new way
to investigate the stability conditions via exceptional sequences.
Basing on this idea, Dimitrov and Katzarkov in \cite{DK} defined a
$\sigma$-exceptional sequence for a given stability condition
$\sigma$, and proved that for the acyclic triangular quiver $Q$,
there exists a $\sigma$-exceptional triple for each stability
condition $\sigma$ on $\mathcal{C}=D^b(\mod kQ)$. We remark that
this implies the connectedness of the space $\Stab(\mathcal{C})$ by
the transitivity of the complete exceptional triples.

The concept of a t-stability in a triangulated category
$\mathcal{C}$ was first introduced by Gorodentsev--Kuleshov--Rudakov
in \cite{GKR}, which is a generalization of Bridgeland's stability
condition. They established the relations
between t-stabilities and bounded t-structures on $\mathcal{C}$.
Indeed, they achieved a classification of the bounded t-structures
for $\mathcal{C}=D^b(\coh\mathbb{P}^1)$.

In the present paper we study the finest
t-stabilities on $\mathcal{C}$ and apply them to the study of
stability conditions. We give a sufficient and necessary condition
to determine when a t-stability is finest for $\mathcal{C}$
piecewise hereditary. Moreover, for the bounded derived category
$D^{b}(\coh\mathbb{X})$ of coherent sheaves on the weighted
projective line $\X$ of weight type (2), we describe the semistable
subcategories and the final HN triangles for finest t-stabilities in
details. After introducing the notion of a t-exceptional sequence on
$\mathcal{C}$, we show the existence of a t-exceptional triple for
$D^{b}(\coh\mathbb{X})$. Note that there is an equivalence between
$D^{b}(\coh\mathbb{X})$ and the bounded derived category $D^b(\mod
kQ)$ for the acyclic triangular quiver $Q$. We obtain that each
stability condition $\sigma$ on $D^b(\mod kQ)$ admits a
$\sigma$-exceptional triple, which implies the connectedness of the
space $\Stab(D^b(\mod kQ))$ and was first shown by
Dimitrov--Katzarkov in \cite{DK}.

The paper is organized as follows. In Section 2 we briefly introduce
the category $\coh\X$ of coherent sheaves on a weighted projective
line $\X$, and recall the definition and basic results for
exceptional sequences in a triangulated category $\mathcal{C}$.
Section 3 is the main part of this paper. We first recall the
definition of a (finest) t-stability on $\mathcal{C}$, and give a
sufficient and necessary condition for a t-stability to be finest.
Moreover, for the bounded derived category $D^b(\coh\X)$ of coherent
sheaves of weight type (2), we describe the semistable subcategories
as well as the final HN triangles for coherent sheaves. Furthermore,
by introducing the notion of a t-exceptional sequence in
$\mathcal{C}$, we prove that each finest t-stability admits a
t-exceptional triple for $D^b(\coh\X)$. Basing on the results of
t-stabilities, we investigate stability conditions on $\mathcal{C}$
in the sense of Bridgeland in Section 4. We obtain that each
stability condition $\sigma$ admits a $\sigma$-exceptional triple
for the acyclic triangular quiver $Q$, which was first shown by
Dimitrov--Katzarkov \cite{DK}.

Throughout this paper we always assume that $\mathcal{C}$ is an essentially small triangulated category, that is, the isomorphism classes of objects in $\mathcal{C}$ form a set.
Given a set $\mathcal{S}$ of objects in $\mathcal{C}$, we write
$\langle\mathcal{S}\rangle$ for the smallest strictly full
extension-closed subcategory of $\mathcal{C}$ that contains all the
objects in $\mathcal{S}$, and write $\rm{Tr}(\mathcal{S})$ for the
minimal full triangulated subcategory containing $\mathcal{S}$ which is
closed under isomorphisms. For an object $E\in\mathcal{C}$, we use
the simple notation $E^n$ to denote the direct sum of $n$ copies of
$E$. For $E,F\in\mathcal{C}$, we simply write
$\Hom(E,F)=\Hom_{\mathcal{C}}(E,F)$ and $\Ext^n(E,F)=\Hom_{\mathcal{C}}(E,F[n])$.

\section{Preliminaries}

\subsection{Coherent sheaves on a weighted projective line}
Following \cite{GL}, a \emph{weighted projective line} $\X=\X_ k $
over a field $ k $ is given by a weight sequence ${\bf
p}=(p_{1},\ldots, p_{t})$ of positive integers, and a sequence
${\boldsymbol\lambda}=(\lambda_{1},,\ldots, \lambda_{t}) $ of
distinct closed points (of degree $1$) in the projective line
$\mathbb{P}^{1}:=\mathbb{P}^{1}_ k $ which can be normalized as
$\lambda_{1}=\infty, \lambda_{2}=0, \lambda_{3}=1$. More precisely,
let $\bbL=\bbL(\bf p)$ be the rank one abelian group with generators
$\vec{x}_{1}, \ldots, \vec{x}_{t}$ and the relations
\[ p_{1}\vx_1=\cdots=p_{t}\vx_t=:\vec{c},\]
where $\vec{c}$ is called the \emph{canonical element} of
$\mathbb{L}$. Each element $\vx\in\bbL$ has the \emph{normal form}
$\vx=\sum\limits_{i=1}^t l_i\vx_i+l\vc$ with $0\leq l_i\leq p_i-1$
and $l\in\bbZ$. Denote by $S$ the commutative algebra
$$S=S({\bf p},{\boldsymbol\lambda})= k [X_{1},\cdots, X_{t}]/I:
=  k [x_{1}, \ldots, x_{t}],$$ where $I=(f_{3},\ldots,f_{t})$ is the
ideal generated by
$f_{i}=X_{i}^{p_{i}}-X_{2}^{p_{2}}+\lambda_{i}X_{1}^{p_{1}}$ for
$3\leq i\leq t$. Then $S$ is $\mathbb{L}$-graded by setting
$$\mbox{deg}(x_{i})=\vx_i\; \text{ for $1\leq i\leq t$.}$$
Finally, the weighted projective line associated with $\bf p$ and
$\boldsymbol\lz$ is defined to be
$$\X=\X_ k =\rm{Spec}^{\bbL}S,$$
 the spectrum of $\bbL$-graded homogeneous prime ideals of $S$.

The category of coherent sheaves on $\X$ is defined to be the
quotient category
$$\coh\X={\rm mod}^{\mathbb{L}}S/\mbox{mod}_{0}^{\mathbb{L}}S,$$
where ${\rm mod}^{\mathbb{L}}S$ is the category of finitely
generated $\mathbb{L}$-graded $S$-modules, while
$\mbox{mod}_{0}^{\mathbb{L}}S$ is the Serre subcategory of finite
length $\mathbb{L}$-graded $S$-modules. The grading shift gives the
twist $E(\vec{x})$ for every sheaf $E$ and $\vec{x}\in\bbL$.

Moreover, $\coh\X$ is a hereditary abelian category with Serre
duality of the form
\begin{equation}\label{serre duality} D\Ext^1(X, Y)\cong\Hom(Y, X(\vec{\omega})),\end{equation}
where $D=\Hom_ k (-, k )$, and
$\vec{\omega}:=(t-2)\vec{c}-\sum_{i=1}^{t}\vec{x}_{i}\in\bbL$,
called the \emph{dualizing element}. This implies the existence of
almost split sequences in $\coh\X$ with the Auslander--Reiten
translation $\tau$ given by the grading shift with $\vec{\omega}$.

It is known that $\coh\X$ admits a splitting torsion pair $({\rm
coh}_0\X,\vect{\X})$, where ${\rm coh}_0\X$ and $\vect{\X}$ are full
subcategories of torsion sheaves and vector bundles, respectively.
The free module $S$ yields a structure sheaf $\co\in \vect{\X}$, and
each object in $\vect{\X}$ has a finite filtration by line bundles,
that is, sheaves of the form $\co(\vec{x})$. Moreover, for any
$\vx,\vy\in\bbL$ we have $$\Hom(\co(\vx),\co(\vy))\cong
S_{\vy-\vx}.$$
 The subcategory ${\rm coh}_0\X$ admits ordinary simple
sheaves $S_{\lambda}$ for each $\lambda\in\mathbb{H}_ k
:=\mathbb{P}^{1}_ k \backslash\{\lz_1,\ldots,\lz_t\}$ and
exceptional simple sheaves $S_{i,j}$ for $1\leq i\leq t$ and $0\leq
j\leq p_i-1$. For any line bundle $L$, $S_{\lambda}$ is determined
by the exact sequence $$0\longrightarrow L\xrightarrow{x_2-\lambda
x_1^2}L(\vc)\longrightarrow S_{\lambda}\longrightarrow 0.$$ If we denote by $S_{i,L}$ the
unique exceptional simple sheaf satisfying that $\Hom(L,S_{i,L})\neq
0$, then $S_{i,L}$ fits into the following exact sequences
$$0\longrightarrow L(-\vx_i)\stackrel{x_i}{\longrightarrow}
L\longrightarrow S_{i,L}\longrightarrow 0.$$
Moreover, the nonzero extensions between these simple sheaves are
given by
$$\Ext^1(S_{\lambda}, S_{\lambda})\cong k (\lambda),\; \Ext^1(S_{i,j}, S_{i,j'})\cong k \;\text{ for $j'\equiv j-1$ (mod\, $p_i$),}$$
 where $ k (\lambda)$ denotes the finite extension of $ k $ with $[ k (\lambda): k ]$ the degree of $\lambda$.
For each simple sheaf $S$ and $n\geq 1$, there is a unique sheaf
$S^{(n)}$ with length $n$ and top $S$, which is uniserial. Indeed,
the sheaves $S^{(n)}$ form a complete set of indecomposable objects
in ${\rm coh}_0\mbox{-}\X$. For convenience, we also use the
notation $S_{i,l}$ for $1\leq i\leq t$ and $l\in\bbZ$ to denote the
simple sheaf $S_{i,j}$ with $j\equiv l$ (mod $p_i$) and $0\leq j\leq
p_i-1$.

The following result establishes a close relation between the
weighted projective lines and the canonical algebras $\Lambda({\bf
p}, {\boldsymbol\lambda})$ introduced by Ringel:

\begin{prop}\label{canonical tilting} {\rm(}\cite[Prop. 4.1]{GL}{\rm)} There exists a canonical tilting sheaf $T_{\rm{can}}=\bigoplus\limits_{0\leq \vx\leq \vc}\co(\vx)$ in $\coh\X$ with endomorphism algebra isomorphic to the canonical algebra $\Lambda:=\Lambda({\bf p}, {\boldsymbol\lambda})$. In particular, there is a derived equivalence $D^b(\coh\X)\cong D^b(\mod \Lambda)$.
\end{prop}

\subsection{Exceptional sequence}

In this subsection we recall some basic results on exceptional
sequences on a triangulated category $\mathcal{C}$.

\begin{defn} An object $E$ in $\mathcal{C}$ is called exceptional if $\Hom(E,E[n])\cong \delta_{n,0}k$. An ordered collection of exceptional objects $(E_0,E_1,\cdots,E_n)$ is called an exceptional sequence if $\Hom^{\bullet}(E_j,E_i)=0$ for $i<j$; and it is further called an Ext-exceptional sequence if $\Hom^{\leq 0}(E_i,E_j)=0$ for $i< j$.
\end{defn}

Let $(E,F)$ be an exceptional pair. Recall that the left mutation
$\mathcal{L}_{E}(F)$ and right mutation $\mathcal{R}_{F}(E)$ are
defined by the following distinguished triangles (see for example
\cite{EM}):
$$\mathcal{L}_{E}(F)\longrightarrow\Hom^{\bullet}(E,F)\otimes E\longrightarrow F,$$
$$E\longrightarrow \rm{D}\Hom^{\bullet}(E,F)\otimes F\longrightarrow \mathcal{R}_{F}(E);$$
where $D=\Hom_k(-,k)$, $V[l]\otimes E$ (with $V$ a vector space)
denotes an object isomorphic to the direct sum of $\dim V$ copies of
the object $E[l]$.

\begin{lem} Let $(E_1,E_{2},\cdots,E_n)$ be an Ext-exceptional
sequence in $\mathcal{C}$ with $\Ext^{j}(E_i,E_{i+1})\neq 0$ for
some $1\leq i<n,$ and $j\in\mathbb{Z}$. Then $(E_1,\cdots,E_{i-1},
\mathcal{R}_{E_{i+1}}(E_{i})[-j])$ and
$(\mathcal{L}_{E_{i}}(E_{i+1})[j],E_{i+2}, \cdots,E_n)$ are
Ext-exceptional sequences.
\end{lem}

\begin{proof} We only prove for the right mutation case, the proof for the left mutation one is similar. Obviously, it suffices to show that for any $1\leq l<i-1, (E_l,\mathcal{R}_{E_{i+1}}(E_{i})[-j])$ is an Ext-exceptional pair.

By the definition of right mutation we know that
$(E_l,\mathcal{R}_{E_{i+1}}(E_{i})[-j])$ is an exceptional pair.
Observe that $\Hom^{\bullet}(E_i,E_{i+1})=\Hom(E_i,E_{i+1}[j])$. By
applying $\Hom^{\leq 0}(E_l,-)$ to the triangle $$E_i\longrightarrow
\rm{D}\Hom(E_i,E_{i+1}[j])\otimes E_{i+1}[j]\longrightarrow
\mathcal{R}_{E_{i+1}}(E_i)$$ we obtain that  $\Hom^{\leq
0}(E_l,\mathcal{R}_{E_{i+1}}(E_{i})[-j])=0$. Hence,
$(E_l,\mathcal{R}_{E_{i+1}}(E_{i})[-j])$ is Ext-exceptional.
\end{proof}

The following result is well-known for the triangulated category
$D^b(\coh\X)$:

\begin{lem}{\rm(}\cite[Lem. 3.2.4]{HM}{\rm)} For any exceptional objects $E,F$ in $D^b(\coh\X)$, there exists at most one integer $n$, such that $\Hom(E,F[n])\neq 0$.
\end{lem}

In this paper we mainly focus on the bounded derived category
$D^b(\coh\X)$ for the weighted projective line $\X$ of weight type
(2). For this special case, we can say more on exceptional
sequences:

\begin{prop}
\label{exceptional pair and triple} Assume $\X$ has weight type (2).
Up to degree shift, all the exceptional pairs in $\coh\X$ are given
by
$$(\co,\co(\vc)),(\co,\co(\vec{x}_1)),(\co,S_{1,0}),(S_{1,1},\co);$$
and all the complete exceptional sequences in $\coh\X$ have the
following forms
$$(\co,\co(\vc),S_{1,0}),(\co,\co(\vec{x}_1),\co(\vc)),(\co,S_{1,0},\co(\vec{x}_1)),(S_{1,1},\co,\co(\vc)).$$
\end{prop}

\begin{proof} The first asssertion follows from the facts $\Hom(\co(\vx),\co(\vy))\cong S_{\vy-\vx}$ and
$\Hom(\co, S_{1,i})\cong\delta_{i0}k$ for $i=0,1$ and the Serre
duality (\ref{serre duality}). The second one is an immediate
consequence.
\end{proof}

\begin{rem}\label{excep}Let $(E_0,E_1,\cdots,E_n)$ be an exceptional sequence on $\coh\mathbb{X}$ for any weighted projective line $\X$. It is known that
$(E_0[k_0],E_1[k_1],\cdots,E_n[k_n])$ is again an exceptional
sequence for any $k_i\in\mathbb{Z}$. Moreover, for $1\leq i,j\leq n$, there exists at most
one integer $k_{ij}$ satisfying $\Ext^{1}(E_i,E_j[k_{ij}])\neq 0$.
In case that $\X$ is of weight type $(2)$, there are no
Hom-orthogonal exceptional pairs by Proposition \ref{exceptional
pair and triple}. Hence, such $k_{ij}$ always exists. In this case,
set
$\widetilde{E_0}=E_0,\widetilde{E_1}=E_1[k_{01}],\widetilde{E_2}=E_2[k_{01}+k_{12}]$.
Then $(\widetilde{E_0},\widetilde{E_1},\widetilde{E_2})$ is a
complete exceptional sequence on $D^b(\coh\mathbb{X})$ satisfying
$\widetilde{E_0}\in\coh\mathbb{X},\Ext^{1}(\widetilde{E_0},\widetilde{E_1})\neq
0$ and $\Ext^{1}(\widetilde{E_1},\widetilde{E_2})\neq 0$. Therefore,
each Ext-exceptional sequence in $D^b(\coh\mathbb{X})$ has the form
\begin{equation}\label{ext-exc law} (\widetilde{E_0}[k_0],\widetilde{E_1}[k_1],\widetilde{E_2}[k_2]) \quad\text{with}\quad k_0\geq k_1\geq k_2.
\end{equation}
\end{rem}

\begin{example}\label{kij}  Assume $\X$ has weight type (2). The exceptional sequences $(\widetilde{E_0},\widetilde{E_1},\widetilde{E_2})$ defined as above can be explicitly listed as follows:
\begin{itemize}
 \item [(i)] $(E_0,E_1,E_2)=(\co,\co(\vx_1),\co(\vec{c}))\;\Rightarrow\; (\widetilde{E_0},\widetilde{E_1},\widetilde{E_2})=(\co,\co(\vx_1)[-1],\co(\vec{c})[-2])$;
 \item [(ii)] $(E_0,E_1,E_2)=(\co,\co(\vec{c}),S_{1,0})\;\Rightarrow\; (\widetilde{E_0},\widetilde{E_1},\widetilde{E_2})=(\co,\co(\vec{c})[-1],S_{1,0}[-2])$;
 \item [(iii)] $(E_0,E_1,E_2)=(\co,S_{1,0},\co(\vx_1))\;\Rightarrow\; (\widetilde{E_0},\widetilde{E_1},\widetilde{E_2})=(\co,S_{1,0}[-1],\co(\vx_1)[-1])$;
 \item [(iv)] $(E_0,E_1,E_2)=(S_{1,1},\co,\co(\vec{c}))\;\Rightarrow\; (\widetilde{E_0},\widetilde{E_1},\widetilde{E_2})=(S_{1,1},\co,\co(\vec{c})[-1])$.
 \end{itemize}
\end{example}

\begin{lem}\label{12to02}
 Assume $\X$ has weight type (2). Let $(E_0,E_1,E_2)$ be an Ext-exceptional sequence in $D^b(\coh\X)$. If $\Hom(E_i,E_{i+1}[1])=0$ for $i=0$ or 1, then $\Hom(E_0,E_2[1])=0$.
\end{lem}
\begin{proof} Assume $(E_0,E_1,E_2)$ has the form (\ref{ext-exc law}). If $\Hom(E_1,E_2[1])=0$, then $k_1>k_2$. It follows that $k_0>k_2$. Hence we have $\Hom(E_0,E_2[1])=0$ by Example \ref{kij}. Similarly, if $\Hom(E_0,E_1[1])=0$, we also have $\Hom(E_0,E_2[1])=0$.
\end{proof}
\section{Finest t-stability}

In this section, we first recall the definition and some basic
results of (finest) t-stability for a triangulated category
$\mathcal{C}$ in the sense of Gorodentsev--Kuleshov--Rudakov, and
describe a sufficient and necessary condition for a t-stability to
be finest. Furthermore, for the bounded derived category of coherent
sheaves on the weighted projective line of weight type (2), we
describe the semistable subcategories, as well as the final HN
triangles for certain coherent sheaves. Finally, we introduce the
notion of t-exceptional triples and prove their existence.

\subsection{t-stability}
\begin{defn}{\rm(}\cite[Def. 3.1]{GKR}{\rm)}
Let $\Phi$ be a linearly ordered set and
$\Pi_\varphi\subset\mathcal{C}$ be a strictly full extension-closed
non-empty subcategory for each $\varphi\in\Phi$. The pair
$(\Phi,\{\Pi_\varphi\}_{\varphi\in\Phi})$ is called a t-stability if

\begin{itemize}
  \item [(i)] the grading shift functor $X\mapsto X[1]$ acts on $\Phi$ as a non-decreasing automorphism, that is, there is a bijection $\tau_{\Phi}\in \;{\rm{Aut}}\;\Phi$ such that $\Pi_{\tau_{\Phi}(\varphi)}=\Pi _{\varphi}[1]$ and $\tau_{\Phi}(\varphi)\geq\varphi$ for all $\varphi$;
  \item [(ii)] ${\rm{Hom}}^{\leq 0}(\Pi _{\psi},\Pi_{\varphi})=0$ for all $\psi>\varphi$ in $\Phi$;
  \item [(iii)] for each non-zero object $X\in\mathcal{C}$, there exists a sequence of triangles
\begin{equation}\label{HN-filt}
\xymatrix { 0=E_0\ar[rr] &&E_1\ar[dl]\ar[rr]&&E_2\ar[dl]\ar[r]&\cdots\ar[r]&E_{n-1}\ar[rr]&&E_n=X,\ar[dl]\\
&A_1\ar@{-->}[ul]&&A_2\ar@{-->}[ul]&&&&A_n\ar@{-->}[ul] }
\end{equation}
where $A_j\in\Pi_{\varphi_j}$ with $\varphi _i>\varphi
_j,\forall\,1\leq i<j\leq n$.
\end{itemize}
\end{defn}

It has been shown in \cite{GKR} that the decomposition
$\eqref{HN-filt}$ for each $X$ is unique up to isomorphism, which
is known as the \emph{Harder-Narasimhan filtration} (\emph{HN
filtration} for short) of $X$. Define $\varphi^{-}(X):=\varphi_n$
and $\varphi^{+}(X):=\varphi_1$. Then $X\in\Pi_{\varphi}$ if and
only if $\varphi^{-}(X)=\varphi^{+}(X)=\varphi=:\varphi(X)$. The
categories $\Pi _\varphi$ are called the semistable subcategories of
the t-stability $(\Phi,\{\Pi_\varphi\}_{\varphi\in\Phi})$. Note that
each $\Pi _\varphi$ is closed under extensions and direct summands,
but, in general, not abelian. The nonzero objects in $\Pi _\varphi$ are
said to be semistable of phase $\varphi$, while the minimal objects
are said to be stable. For any interval $I\subseteq \Phi$,
$\Pi_{I}$ is defined to be the extension-closed subcategory of
$\mathcal{C}$ generated by the subcategories $\Pi_{\varphi},
\;\varphi\in I$.

\begin{prop}\label{t-stability-t-structure}{\rm(}\cite[Cor. 5.2]{GKR}{\rm)}
 Let $(\Phi,\{\Pi_{\varphi}\}_{\varphi\in\Phi})$ be a t-stability on $\mathcal{C}$. Then each $\varphi\in\Phi$ determines two t-structures $\cal{A}_{\varphi},\cal{B}_{\varphi}$ such that
$$\cal{A}^{\geq 0}_{\varphi}=\langle\Pi_{\psi}\mid\psi\leq\tau_{\Phi}(\varphi)\rangle,\quad\cal{A}^{\leq 0}_{\varphi}=\langle\Pi_{\psi}\mid\psi>\varphi\rangle;$$
$$\cal{B}^{\geq 0}_{\varphi}=\langle\Pi_{\psi}\mid\psi<\tau_{\Phi}(\varphi)\rangle,\quad\cal{B}^{\leq 0}_{\varphi}=\langle\Pi_{\psi}\mid\psi\geq\varphi\rangle.$$
Moreover, the corresponding hearts are given by
$\Pi_{(\varphi,\tau_{\Phi}(\varphi)]}$ and
$\Pi_{[\varphi,\tau_{\Phi}(\varphi))}$, respectively.
\end{prop}

Using a proof similar to that of \cite[Lem.~3.4]{TB}, we have the
following result.
\begin{lem}\label{similar result as stab condition} Let $(\Phi,\{\Pi_{\varphi}\}_{\varphi\in\Phi})$ be a t-stability on $\mathcal{C}$. Assume there exists some $\varphi_0\in\Phi$ such that all the objects in the triangle
$A\to X\to B$ are nonzero in
$\Pi_{[\varphi_0,\tau_{\Phi}(\varphi_0)]}$, then
$$\varphi^{+}(A)\leq \varphi^{+}(X)\;\text{ and }\;\varphi^{-}(X)\leq \varphi^{-}(B).$$
\end{lem}

\subsection{Finest t-stability}
We recall the definition of a partial order for t-stabilities given in \cite{GKR}.
\begin{defn} Let $(\Phi,\{\Pi_{\varphi}\}_{\varphi\in\Phi}),(\Psi,\{P_{\psi}\}_{\psi\in\Psi})$ be t-stabilities on $\mathcal{C}$ and let the grading shift functor act on $\Phi,\Psi$ by automorphisms $\tau_{\Phi},\tau_{\Psi}$ respectively. We say that the t-stability $(\Phi,\{\Pi_{\varphi}\}_{\varphi\in\Phi})$ is finer than $(\Psi,\{P_{\psi}\}_{\psi\in\Psi})$ if there exists a surjective map $r:\Phi\rightarrow\Psi$ such that
\begin{itemize}
  \item [(i)] $r\tau_{\Phi}=\tau_{\Psi}r$;
  \item [(ii)] $\varphi'>\varphi''$ implies $ r(\varphi')\geq r(\varphi'')$;
  \item [(iii)] for any $\psi\in\Psi$, $P_{\psi}=\langle\Pi_\varphi\mid\varphi\in r^{-1}(\psi)\rangle$.
\end{itemize}
\end{defn}

Minimal elements with respect to this partial order will be called
the finest t-stabilities. In this subsection we will give an
equivalent description of finest t-stabilities for a triangulated
category $\mathcal{C}$ which is piecewise hereditary. That is, we
assume that there exists a hereditary abelian category $\mathcal{H}$
such that there is a triangulated equivalence $\mathcal{C}\cong
D^b(\mathcal{H})$ in the rest of this subsection.

\begin{lem}\label{strictly increasing} Let $(\Phi,\{\Pi_{\varphi}\}_{\varphi\in\Phi})$ be a finest t-stability on $\mathcal{C}$. Then $\tau _{\Phi}$ is strictly increasing.
\end{lem}
\begin{proof} Suppose there exists $\psi\in\Phi$ such that $\tau _{\Phi}(\psi)=\psi$. Define $\Pi_{\psi_i}=\Pi_\psi\cap(\mathcal{H}[i])$ for all $i\in\mathbb{Z}$. Then $\Pi_{\psi}=\cup_{i\in\mathbb{Z}}\Pi_{\psi_i}$. Set
$$\Psi:=\big(\Phi\setminus\{\psi\}\big)\bigcup\{\psi_{i}\mid i\in\bbZ\}.$$
Define an automorphism $\tau_{\Psi}$ on $\Psi$ by setting
$$\tau_{\Psi}(\psi_{i})=\psi_{i+1},\forall\; i\in\bbZ;\;\tau_{\Psi}(\varphi)=\tau_{\Phi}(\varphi),\forall
\varphi\in\Phi\setminus\{\psi\}.$$ We define a linear order on
$\Psi$ by keeping the order relations in $\Phi\setminus\{\psi\}$ and
adding new order relations
$$\varphi_1<\psi_i<\psi_{i+1}<\varphi_2\;\;(\forall\,i\in\bbZ)$$
 whenever $\varphi_1<\psi<\varphi_2$ with $\varphi_1,\varphi_2\in\Phi\setminus\{\psi\}$.
Then $(\Psi,\{\Pi_{\varphi}\}_{\varphi\in\Psi})$ is a t-stability
which is strictly finer than
$(\Phi,\{\Pi_{\Phi}\}_{\varphi\in\Phi})$, a contradiction.
\end{proof}

\begin{lem}\label{hom non trivial} Let $(\Phi,\{\Pi_{\varphi}\}_{\varphi\in\Phi})$ be a t-stability on $\mathcal{C}$. If there exist some $\psi\in\Phi$ and $X,Y\in\Pi_{\psi}$ such that $\Hom(X,Y)=0$, then $(\Phi,\{\Pi_{\varphi}\}_{\varphi\in\Phi})$ is not finest.
\end{lem}

\begin{proof} By Lemma \ref{strictly increasing}, we can assume that $\tau _{\Phi}$ is strictly increasing. Denote by $\mathcal{A}:=
\Pi_{(\tau_{\Phi}^{-1}(\psi),\psi]}$. By Proposition
\ref{t-stability-t-structure}, $\mathcal{A}$ is abelian. Let
$$\mathcal{F}=\{Z\in\mathcal{A}\mid \Hom(X,Z)=0\};\quad \mathcal{T}=\{W\in\mathcal{A}\mid \Hom(W,Z)=0, \forall Z\in\mathcal{F}\}.$$
Then $(\mathcal{T},\mathcal{F})$ forms a torsion pair in
$\mathcal{A}$. By definition, $Y\in\mathcal{F}$ and
$X\in\mathcal{T}$. Moreover, we have
$\Pi_{(\tau_{\Phi}^{-1}(\psi),\psi)}\subseteq\mathcal{F}$, which
ensures that  $\mathcal{T}\subseteq\Pi_{\psi}$.

Set
$$\Psi:=\big(\Phi\setminus\{\tau_{\Phi}^{n}(\psi)\mid n\in\bbZ\}\big)\bigcup\{\psi_{1,n},\psi_{2,n}\mid n\in\bbZ\}.$$
Define an automorphism $\tau_{\Psi}$ on $\Psi$ by setting
$$\tau_{\Psi}(\psi_{i,n})=\psi_{i,n+1},i=1,2;\tau_{\Psi}(\varphi)=\tau_{\Phi}(\varphi),\forall
\varphi\in\Phi\setminus\{\tau_{\Phi}^{n}(\psi)\mid n\in\bbZ\}.$$ We define a linear order on $\Psi$ by keeping the order relations in
$\Phi\setminus\{\tau_{\Phi}^{n}(\psi)\mid n\in\bbZ\}$ and adding new
order relations
$$\psi_{1,n}<\psi_{2,n}<\psi_{1,n+1}\;\;(\forall\,n\in\bbZ)$$ and
$$\tau_{\Psi}^n(\varphi_1)<\psi_{1,n}<\psi_{2,n}<\tau_{\Psi}^n(\varphi_2)$$
 whenever $\varphi_1<\psi<\varphi_2$ with $\varphi_1,\varphi_2\in\Phi\setminus\{\tau_{\Phi}^{n}(\psi)\mid n\in\bbZ\}$.

Define $\Pi_{\psi_{1,0}}=\Pi_{\psi}\bigcap\mathcal{F}$,
$\Pi_{\psi_{2,0}}=\Pi_{\psi}\bigcap\mathcal{T}=\mathcal{T}$, and set
$\Pi_{\psi_{i,n}}=\Pi_{\psi_{i,0}}[n]$ for $n\in\bbZ$, $i=1,2$. We
claim that $\Pi_{\psi_{1,0}}, \Pi_{\psi_{2,0}}$ satisfy the
following conditions:
\begin{itemize}
\item[(i)] $\Hom(\Pi_{\psi_{2,0}}, \Pi_{\psi_{1,0}})=0$;
\item[(ii)] For each element $Z\in\Pi_{\psi}$, there exists a unique exact sequence $0\to Z_2\to Z\to Z_1\to 0$, where $Z_i\in\Pi_{\psi_{i,0}}$ for $i=1,2$.
\end{itemize}
In fact, since $(\mathcal{T},\mathcal{F})$ is a torsion pair, the
statement (i) follows immediately from the fact
$\Hom(\mathcal{T},\mathcal{F})=0$. For the second statement, note
that $Z$ has a unique decomposition $0\to Z_2\to Z\to Z_1\to 0$ with
$Z_2\in\mathcal{T}$ and $Z_1\in\mathcal{F}$. Thus, it suffices to
show that $Z_1\in\Pi_{\psi}$. This follows from the fact that
$\Hom(Z,\Pi_{\varphi})=0$ for any
$\varphi\in(\tau_{\Phi}^{-1}(\psi),\psi)$. Thus, $(\Psi,
\{\Pi_{\varphi}\}_{\varphi\in\Psi})$ is a t-stability on $\mathcal{C}$, which is strictly finer than
$(\Phi,\{\Pi_{\varphi}\}_{\varphi\in\Phi})$. We are done.
\end{proof}

The following result gives an equivalent description of the finest
t-stabilities.

\begin{thm}\label{iff cond for finest} A t-stability $(\Phi,\{\Pi_{\varphi}\}_{\varphi\in\Phi})$ on  $\mathcal{C}$ is finest if and only if for any $\varphi\in\Phi$ and $X,Y\in\Pi_{\varphi}$,  $\Hom(X,Y)\neq 0\neq \Hom(Y,X)$.
\end{thm}

\begin{proof} The $\lq\lq$if" part follows from \cite[Prop.~5.5]{GKR}, while the $\lq\lq$only if" part follows from Lemma \ref{hom non trivial}.
\end{proof}

\subsection{Semistable subcategories}

From now onwards, let $\mathbb{X}$ be the weighted projective line
over $k$ of weight type (2) and $\coh\mathbb{X}$ be the category of
coherent sheaves on $\mathbb{X}$. Let
$\cal{D}=D^{b}(\coh\mathbb{X})$ be the bounded derived category of
$\coh\mathbb{X}$. Similar to the case of the projective line
$\mathbb P^1$, each t-stability in $\cal{D}$ can be refined to a
finest one. In the following, we always fix a finest t-stability
$(\Phi,\{\Pi_{\varphi}\}_{\varphi\in\Phi})$ on $\cal{D}$.

The following result is an immediate consequence of Theorem \ref{iff
cond for finest}.
\begin{cor}\label{quasimple} If $S_{1,0},S_{1,1}$ are semistable, then $\varphi(S_{1,0})\neq\varphi(S_{1,1})$.
\end{cor}

\begin{proof}
This follows from Theorem \ref{iff cond for finest} and
$\Hom(S_{1,0}, S_{1,1})=0=\Hom(S_{1,1}, S_{1,0})$.
\end{proof}

\begin{cor}\label{at most one semistable} At most one of $S^{(2)}_{1,0}$ or $S^{(2)}_{1,1}$ is semistable.
\end{cor}
\begin{proof} Suppose that both of $S^{(2)}_{1,0},S^{(2)}_{1,1}$ are semistable. Since $\Hom(S^{(2)}_{1,0},S^{(2)}_{1,1})\neq 0\neq \Hom(S^{(2)}_{1,1},S^{(2)}_{1,0})$, we obtain that $\varphi(S^{(2)}_{1,0})=\varphi(S^{(2)}_{1,1})$.
Now consider the Auslander--Reiten sequence $$0\to S^{(2)}_{1,i}\to
S^{(3)}_{1,i+1}\oplus S_{1,i}\to S^{(2)}_{1,i+1}\to 0$$ for $i=0,1$.
Thus, $S_{1,0}, S_{1,1}$ are semistable and
$\varphi(S_{1,0})=\varphi(S^{(2)}_{1,0})=\varphi(S_{1,1})$,
contradicting to Corollary \ref{quasimple}.
\end{proof}

\begin{lem}\label{S and Sn} Let $S$ be the torsion sheaf $S_{\lambda}$ for some ordinary point $\lambda\in\mathbb{H}_ k$ or $S_{1,i}^{(2)}$ for $i=0,1$. Then $S\in \Pi_{\varphi}$ if and only if  $S^{(n)}\in\Pi_{\varphi}$ for any $n$.
\end{lem}
\begin{proof} Assume that $S\in \Pi_{\varphi}$, then by $S^{(n)}\in \langle S\rangle$ we know that $S^{(n)}\in\Pi_{\varphi}$.
Now assume $S^{(n)}\in\Pi_{\varphi}$ for some $n$. Consider the
exact sequence $0\to S^{(n)}\to S\oplus S^{(2n-1)}\to S^{(n)}\to 0$.
Since $\Pi_{\varphi}$ is closed under extensions and direct
summands, we conclude that $S\in \Pi_{\varphi}$.
\end{proof}

\begin{lem}\label{not semistable} For any $n\in\mathbb{N}$ and $i=0,1$, $S^{(2n+1)}_{1,i}$ is not semistable.
\end{lem}
\begin{proof}

Suppose there exist $n\in\mathbb{N}$, $i=0$ or
1, such that $S^{(2n+1)}_{1,i}\in\Pi_{\psi}$ for some $\psi\in\Phi$.
From the exact sequence $0\to S^{(2n+1)}_{1,i}\to
S^{(2n-1)}_{1,i}\oplus S^{(2n+3)}_{1,i}\to S^{(2n+1)}_{1,i}\to 0$
and by induction, we get $S_{1,i}^{(2m-1)}\in\Pi_{\psi}$ for any
$m\in\mathbb{N}$. It follows that $S^{(2)}_{1,i}$ and
$S^{(2)}_{1,i+1}$ are both semistable. Note that there are two
sequences of non-zero morphism $S^{(3)}_{1,i+1}\to
S^{(2)}_{1,i+1}\to S_{1,i+1}$ and $S_{1,i+1}\to S^{(2)}_{1,i}\to
S^{(3)}_{1,i+1}$. Hence
$\varphi(S^{(2)}_{1,i})=\varphi(S^{(2)}_{1,i+1})=\psi$. Now from the
exact sequence $0\to S^{(2)}_{1,i+1}\to S_{1,i+1}\oplus
S^{(3)}_{1,i}\to S^{(2)}_{1,i}\to 0$ we obtain that
$\varphi(S_{1,i+1})=\psi=\varphi(S_{1,i})$, a contradiction to Lemma
\ref{quasimple}.

\end{proof}

The following is a characterization for semistable subcategories for
$\cal{D}$.

\begin{thm}\label{semistable subcategory} Each semistable subcategory of $\cal{D}$ has the form $\langle E[j]\rangle$, where $j\in\mathbb{Z}$ and $E$ is a coherent sheaf satisfying that $\End(E)$ is a division algebra.
\end{thm}

\begin{proof} Combining Lemma \ref{S and Sn} with Lemma \ref{not semistable}, the possible simple objects (up to shift) for each semistable category are line bundles or simple sheaves or $S_{1,i}^{(2)}$ for $i=0,1$. We observe that these are all the sheaves in $\coh\X$ whose endomorphism algebras are division algebras. Note that for any two such sheaves $X,Y$ and any integers $m,n$, we have $\Hom(X[m],Y[n])=0$ or $\Hom(Y[n],X[m])=0$. Then by Theorem \ref{iff cond for finest}, $X[m]$ and $Y[n]$ have different phases. Hence, each semistable subcategory contains a unique simple object. We are done.
\end{proof}

\subsection{Final HN triangles}

For any object $X\in D^{b}(\coh\mathbb{X})$, we use $\Delta_{X}$ to
denote the \emph{final HN triangle} $E_{n-1}\to X\to A_n$ in the HN
filtration (\ref{HN-filt}) and call $A_n$ the \emph{final HN factor}
of $X$. In this subsection, we will investigate the possible forms
of $\Delta_{X}$ for indecomposable $X$. If $X$ itself is semistable,
then $\Delta_{X}$ has trivial form. For this reason, we always
assume that $X$ is not semistable and $\Delta_{X}$ has non-trivial
form in the following.

\begin{lem}\label{no-existence} Let $Y\rightarrow X\rightarrow Z$ be the final HN
triangle of $X$. Then there does not exist a semistable object $W$
satisfying $\Hom(Y,W)\neq 0$ and $\varphi(W)<\varphi(Z)$.
\end{lem}
\begin{proof} Suppose there exists such a semistable object $W$.
Then $\varphi^{-}(Y)\leq\varphi(W)<\varphi(Z)$,
which is a contradiction.
\end{proof}

\begin{lem}\label{HN-factor} Assume $Z_1,Z_2$ are semistable with $\varphi(Z_1)<\varphi(Z_2)$. If $\Hom(X,Z_i)\neq 0$, for $i=1,2$, then $Z_2$ is not the HN factor of $X$.
\end{lem}

\begin{proof} Suppose that $\Delta_X$ has the form $Y\to  X\to Z_2^k$ for some $k\in\bbn$. It follows that
$\varphi^{-}(Y)>\varphi(Z_2)>\varphi(Z_1)$. Hence,
$\Hom(Y,Z_1)=0=\Hom(Z_2,Z_1)$, which implies that $\Hom(X,Z_1)=0$, a
contradiction.
\end{proof}

The following result is very useful in a triangulated category, but
we could not find a proof in the literature. We include the proof
suggested by Nan Gao.

\begin{lem}\label{kernel and cokernel}
Let $f:X\to Y$ be a morphism in a hereditary abelian category
$\mathcal{H}$. Then $f$ fits into the following triangle in the
bounded derived category $D^b(\mathcal{H})$:
$$X\stackrel{f}{\longrightarrow}Y\longrightarrow\coker(f)\oplus\ker(f)[1].$$
\end{lem}

\begin{proof} Consider the following commutative diagram in $\mathcal{H}$:
$$\xymatrix@C=0.5cm{
 0 \ar[rr] && \ker(f) \ar[rr] && X \ar[rr]^{f} \ar[rd]^{\pi} && Y  \ar[rr]&& \coker(f)  \ar[rr] &&0. \\
 && && & \rm{Im}(f) \ar[ru]^{i} }$$
 By the Octahedral Axiom, we obtain the following commutative diagram of triangles:
 $$\xymatrix@C=0.5cm{
 X \ar[r]^{\pi}\ar@{=}[d] & \rm{Im}(f) \ar[r]\ar[d]^{i} & \ker(f)[1]\ar[d] \\
 X \ar[r]^{f} & Y \ar[r]\ar[d] & \rm{cone}(f).\ar[d] \\
& \coker(f) \ar@{=}[r]  & \coker(f)
 }$$
 Since $\mathcal{H}$ is hereditary, we obtain that $\Hom(\coker(f),\ker(f)[2])=0$. It follows that $\rm{cone}(f)\cong \coker(f)\oplus\ker(f)[1]$. We are done.
\end{proof}

\begin{lem}\label{not x}
Let $X$ be an indecomposable coherent sheaf with $\Delta_X$ of the
form $Y\rightarrow X\xrightarrow{f} Z$. Then $Z\notin\langle
S_{\lambda}\rangle$ for any $\lambda\in\mathbb{H}_ k$.
\end{lem}

\begin{proof}
Suppose that $Z\in \langle S_{\lambda}\rangle$
for some $\lambda\in\mathbb{H}_ k$. Consider the exact sequence in
$\coh\X$: $$0\longrightarrow\ker(f)\longrightarrow X\stackrel{f}{\longrightarrow} Z\longrightarrow\coker(f)\longrightarrow 0.$$ It
follows from Lemma \ref{kernel and cokernel} that
$Y=\ker(f)\oplus\coker(f)[-1]$. If $X$ is a line bundle, then
$\ker(f)$ is a line bundle, hence $\Hom(\ker(f), Z)\neq 0$, a
contradiction. If $X$ is a torsion sheaf, then $X\in\langle
S_{\lambda}\rangle$, it follows that $X$ is semistable, again a
contradiction. This finishes the proof.
\end{proof}

\begin{lem}\label{isotropy} Let $X$ be an indecomposable object in $\cal{D}$. Then the final HN factor of $X$ has the form $E^n[j]$, where $j\in\mathbb{Z}, n\in\mathbb{N}$, $E$ is a line bundle or $E\in\{S_{1,i}, S_{1,i}^{(2m)} | i=0,1;\;m\in\mathbb{N}\}.$ Moreover, if $E\in\langle S_{1,i}^{(2)} \rangle$ for $i=0$ or $1$, then $n=1$.
\end{lem}

\begin{proof} Without loss of generality, we assume that $X\in\coh\mathbb{X}$ and $\Delta_X$ has the form $Y\rightarrow X\xrightarrow{f} Z$. By Theorem \ref{semistable subcategory} and Lemma \ref{not x}, the final HN factor $Z$ of $X$ lies in $\langle E[j]\rangle$, where $j=0$ or $1$, $E$ is a line bundle or $E\in\{S_{1,i}, S_{1,i}^{(2)} | i=0,1\}.$ If $E$ is a line bundle or $S_{1,i},i=0,1$, then we are done. Otherwise, $E=S_{1,i}^{(2)}$ for $i=0$ or 1. Similar to the proof of Lemma \ref{not x}, we know that $X\in\langle S_{1,0},S_{1,1}\rangle$. It follows that $\Hom(X, S_{1,i}^{(2)})\neq 0$ if and only if $\Hom(X, S_{1,i}^{(2)}[1])\neq 0$. By Lemma \ref{HN-factor}, $S_{1,i}^{(2)}[1]$ is not the final HN factor of $X$. Hence $Z\in\langle S_{1,i}^{(2)}\rangle$. We claim that $Z$ is indecomposable. Indeed, consider the following exact sequence in $\coh\X$: $$0\longrightarrow\ker(f)\longrightarrow X\stackrel{f}{\longrightarrow} Z\longrightarrow\coker(f)\longrightarrow 0.$$
Then by Lemma \ref{kernel and cokernel} we have $Y\cong
\ker(f)\oplus \coker(f)[-1]$. It follows that
$\ker(f),\coker(f)\in\langle S_{1,i}\rangle$ by $\Hom(Y,Z)=0$.
Moreover, $X$ is indecomposable implies $\ker(f)\in\{0,S_{1,i}\}$.
Hence $\rm{Im}(f)$ is also indecomposable, which ensures
$\coker(f)\in\{0,S_{1,i}\}$. Therefore, $Z$ is indecomposable, i.e.,
$n=1$.
\end{proof}

\begin{thm} For any $X\in\cal{D}$ with $\Delta_{X}:$ $Y\xrightarrow{g} X\xrightarrow{f} Z$, where $Z\in\langle E[j]\rangle$ for some $j\in\bbZ$,
\begin{itemize}
  \item[(i)] if $E= S_{1,i}^{(2)}$ for $i=0$ or 1, then $f$ has the form $X\to S_{1,i}^{(2m)}[j]$ for some $m\geq 1$.
   \item[(ii)] if $E$ is exceptional, then $f$ has the form $X\xrightarrow{\ev} \Hom(X, E[j])\otimes E[j]$.
\end{itemize}
\end{thm}

\begin{proof}
By Theorem \ref{semistable subcategory}, the final HN factor has the
form $Z=E[j]^n$, where $j\in\mathbb{Z}, n\in\mathbb{N}$ and $E$ is a
coherent sheaf satisfying that $\End(E)$ is a division algebra. If
$E= S_{1,i}^{(2)}$ for $i=0$ or 1, then $Z$ is indecomposable by
Lemma \ref{isotropy}, we are done. If else, $E$ is exceptional.
Assume that $f=(f_1,f_2,\cdots,f_n)^{t}$ with $f_i\in\Hom(X, E[j])$.
We claim that $n=\dim\Hom(X, E[j])$ and $f_1,f_2,\cdots,f_n$ form a
basis of $\Hom(X, E[j])$. In fact, if $n<\dim\Hom(X, E[j])$, then
there exists a map $f_0\in\Hom(X, E[j])$ which is not a linear
combination of $\{f_1,f_2,\cdots,f_n\}$. Now consider the following
diagram:
$$\xymatrix{
  Y \ar[r]^{g} &X \ar[d]^{f_0}\ar[r]^{(f_1,f_2,\cdots,f_n)^{t}}& E[j]^n\\
  & E[j]  }
$$
Since $E$ is exceptional, we have $\Hom(E[j],E[j])\cong k$. Thus
$f_0$ can not factor through $(f_1,f_2,\cdots,f_n)^{t}$. It follows
that $f_0\circ g\neq 0$, hence $\Hom(Y, E[j])\neq 0$. This contradicts the
definition of HN-filtration. Now suppose
$n>\dim\Hom(X, E[j])$, or $n=\dim\Hom(X, E[j])$ with
$\{f_1,f_2,\cdots,f_n\}$ linearly dependent, we obtain that the
triangle $Y \xrightarrow{g} X \xrightarrow{(f_1,f_2,\cdots,f_n)^{t}}
E[j]^n\to Y[1]$ is split and $E[j]$ is a direct summand of $Y[1]$.
It follows that $\Hom(Y,E[j-1])\neq 0$, a contradiction. We are
done.
\end{proof}

In the following we describe the explicit forms of the final HN
triangles for special coherent sheaves.

\begin{prop}\label{hn for line bundle} Let $L$ be a line bundle. Then $\Delta_{L}$ has one of the following forms:
\begin{itemize}
  \item[(i)] $L((k+1)\vc)^{k}[-1]\rightarrow L\rightarrow L(k\vc)^{k+1},\quad k>0;$
  \item[(ii)] $L(-k\vc)^{k+1}\rightarrow L\rightarrow L(-(k+1)\vc)^{k}[1],\quad k>0;$
  \item[(iii)] $L((k+1)\vc+\vx_1)^{k}[-1]\oplus S_{1,L(\vx_1)}[-1]\to L\to L(k\vc+\vx_1)^{k+1},\quad k\geq 0;$
   \item[(iv)] $L(-k\vc)\oplus L(-k\vc+\vx_1)^{k}\to L\to L(-k\vc-\vx_1)^{k}[1],\quad k>0;$
   \item[(v)] $L(-\vx_1)\to L\to S_{1,L}.$
\end{itemize}
\end{prop}

\begin{proof}
According to Lemma \ref{isotropy} and using a similar proof for
Lemma \ref{not x}, we know that the possible final HN factors for
$L$ are $S_{1,L}$, $L(\vx)$ or $L(-\vc-\vx)[1]$ for some $\vx>0$.
Now it remains to show that the (co)cone of the evaluation maps have
the desired forms. The first two triangles follow from the embedding
$S_{1,L}^{\bot}=\coh\bbp^1\hookrightarrow \coh\X$ and \cite[Lemma
3.1]{SO}, and the last triangle is trivial. We only show (iii) and
(iv). By applying $\Hom(-,L(k\vc+\vx_1))$ to the exact sequence
$0\to L\to L(\vx_1) \to S_{1,L(\vx_1)}\to 0$, we get an isomorphism
  $\Hom(L,L(k\vc+\vx_1))\cong \Hom(L(\vx_1),L(k\vc+\vx_1))$. Then the statement (iii) follows from the following pullback diagram:
  $$\xymatrix{
  L((k+1)\vc+\vx_1)^{k}[-1] \oplus S_{1,L(\vx_1)}[-1] \ar[r]\ar[d]& L \ar[d]^{x_1}\ar[r]^{\ev}& L(k\vc+\vx_1)^{k+1}\ar@{=}[d]\\
  L((k+1)\vc+\vx_1)^{k}[-1] \ar[r]\ar[d]& L(\vx_1) \ar[d]\ar[r]^{\ev}& L(k\vc+\vx_1)^{k+1}\\
  S_{1,L(\vx_1)} \ar@{=}[r] & S_{1,L(\vx_1)}.  }
$$
Similarly, the statement (iv) follows from the following pullback
diagram:
 $$\xymatrix{
  L(-k\vc)\oplus L(-k\vc+\vx_1)^{k}\ar[r]\ar[d]& L \ar[d]^{x_1}\ar[r]^{\ev}& L(-k\vc-\vx_1)^{k}[1]\ar@{=}[d]\\
  L(-k\vc+\vx_1)^{k+1} \ar[r]\ar[d]& L(\vx_1) \ar[d]\ar[r]^{\ev}& L(-k\vc-\vx_1)^{k}[1]\\
  S_{1,L(\vx_1)} \ar@{=}[r] & S_{1,L(\vx_1)}.  }
$$
 \end{proof}

Let $L$ be a line bundle and let $Y\to X\to L[j]^n$ be the final HN
triangle of $X$ for some $j\in\mathbb{Z}$ and $n\in\mathbb{N}$. If
$L(\vc)[j-1]$ is a direct summand of $Y$, then we say that
$\Delta_{X}$ is of type $(L, L(\vc))$. It follows from the above
proposition that each line bundle $L$ is of type $(L(\vx),
L(\vx+\vc))$ for some $\vx$, or $\Delta_{L}$ has the form
$$L(-\vx_1)\longrightarrow L\longrightarrow S_{1,L} \text{\quad or\quad} S_{1,L(\vx_1)}[-1]\longrightarrow L\longrightarrow L(\vx_1).$$

 \begin{prop}\label{hn for s1i} The final HN triangle of $S_{1,i}, i=0,1$ has one of the following forms:
\begin{itemize}
  \item[(i)] $\co(k\vc+i\vx_1)\to S_{1,i}\to \co(k\vc+(i-1)\vx_1)[1]$ for $k\in\bbZ$;
  \item[(ii)] $S_{1,i}^{(2)}\to S_{1,i}\to S_{1,i+1}[1];$
  \item[(iii)] $S_{1,i+1}[-1] \to S_{1,i}\to S_{1,i+1}^{(2)}.$
\end{itemize}
\end{prop}

\begin{proof}
 By Lemma \ref{isotropy}, all the possibilities of the final HN factor of $S_{1,i}$ are given by $S_{1,i+1}[1]$, $S_{1,i+1}^{(2n)}$ for $n\geq 1$, and $\co(k\vc+(i-1)\vx_1)[1]$ for $k\in\bbZ$. If the final HN factor is $\co(k\vc+(i-1)\vx_1)[1]$ or $S_{1,i+1}[1]$, then it is easy to see that $\Delta_{S_{1,i}}$ has the form in (i),(ii) respectively. Now assume the final HN factor is $S_{1,i+1}^{(2n)}$, then $\Delta_{S_{1,i}}$ is given by $S_{1,i+1}^{(2n-1)}[-1] \to S_{1,i}\to S_{1,i+1}^{(2n)}.$
By $\Hom(S_{1,i+1}^{(2n-1)}[-1], S_{1,i+1}^{(2n)})=0$ we obtain
$n=1$, hence (iii) holds.
\end{proof}

 \begin{prop}\label{hn for s}
The final HN triangle of $S_{1,i}^{(2)}, i=0,1$ has one of the
following forms:
\begin{itemize}
  \item[(i)] $S_{1,i+1} \to S_{1,i}^{(2)}\to S_{1,i};$
    \item[(ii)] $S_{1,i+1}\oplus S_{1,i+1}[-1] \to S_{1,i}^{(2)}\to S_{1,i+1}^{(2)};$
  \item[(iii)] $\co(k\vc+\vc+i\vx_1)\to S_{1,i}^{(2)}\to \co(k\vc+i\vx_1)[1]$ for $k\in\bbZ$;
  \item[(iv)] $\co(k\vc+i\vx_1)\oplus S_{1,i+1}\to S_{1,i}^{(2)}\to \co(k\vc+(i-1)\vx_1)[1]$ for $k\in\bbZ$.
\end{itemize}
\end{prop}

\begin{proof} The proof is similar as in the above corollary. Here we only treat the triangles (ii) and (iv), the other two are trivial.
Note that $\dim\Hom(S_{1,i}^{(2)}, S_{1,i+1}^{(2)})=1$, and both of
kernel and cokernel for any non-zero morphism in
$\Hom(S_{1,i}^{(2)}, S_{1,i+1}^{(2)})$ are given by $S_{1,i+1}$,
then (ii) follows. The triangle (iv) follows from the following
pullback diagram:
$$\xymatrix{
 \co((k-1)\vc+i\vx_1) \ar[r]\ar@{=}[d]& \co(k\vc+(i-1)\vx_1) \ar[r]\ar[d]& S_{1,i+1} \ar[d]\\
  \co((k-1)\vc+i\vx_1) \ar[r]&\co(k\vc+i\vx_1)\ar[d] \ar[r]& S_{1,i}^{(2)}\ar[d]\\
  &S_{1,i} \ar@{=}[r]& S_{1,i}.  }
$$
\end{proof}

\subsection{The distance $d(\Phi)$}

\begin{lem} Let $m$ be the number of semistable line bundles (up to shift). Then $m\geq 2$.
\end{lem}

\begin{proof} Suppose that $m\leq 1$, i.e. there is at most one semistable line bundle, say $L$ if it exists. It follows that $L(\vc)$ and $L(\vx_1)$ are not semistable. By Proposition \ref{hn for line bundle} we obtain the HN filtration of $L(\vc)$ as follows:
\begin{equation}\label{two term of hn for Lc}
\xymatrix{
  L \ar[r] \ar@{.}[rd]& L(\vx_1)\ar@{.}[rd]\ar[r]\ar[d]& L(\vc) \ar[d]\\
  &S_{1,L(\vx_1)}& S_{1,L}.  }
\end{equation}
Then $\Hom(L,S_{1,L})\neq 0$ yields a contradiction. Hence $m\geq
2$.
\end{proof}

Define the \emph{distance} of the t-stability
$(\Phi,\{\Pi_{\varphi}\}_{\varphi\in\Phi})$ by
$$d(\Phi)={\rm{min}}\{\vy-\vx\mid \co(\vx),
\co(\vy)\;{\rm{are\;semistable}}\}.$$ It is well-defined since we
have at least two different line bundles by the above lemma and
$\bbL$ is a total order. Furthermore, we have the following result.
\begin{prop}\label{ss-line-bundle} $d(\Phi)\leq \vc$.
\end{prop}

\begin{proof}
Suppose that $d(\Phi)=\vy>\vc$. Let $L$ be a
semistable line bundle. Then $L(\pm\vx_1)$ and $L(\pm\vc)$ are not
semistable.

We claim that $\Delta_{L(\vx_1)}$ has the form $L\to L(\vx_1)\to
S_{1,L(\vx_1)}$. In fact, by Proposition \ref{hn for line bundle} it
suffices to show that $\Delta_{L(\vx_1)}$ is not of type $(L(\vz),
L(\vz+\vc))$ for any $\vz$. Otherwise, we have $\vz\leq -\vc-\vx_1$
or $\vz\geq \vc+\vx_1$. We can always take $W=L$ to deduce a
contradiction to Lemma \ref{no-existence}. This finishes the proof
of the claim.

Similarly, one can show that $\Delta_{L(\vc)}$ is not of type
$(L(\vz), L(\vz+\vc))$ for any $\vz$. It follows that
$\Delta_{L(\vc)}$ is
$$L(\vx_1)\to L(\vc)\to S_{1,L}\;\text{or}\;
S_{1,L(\vx_1)}[-1]\to L(\vc)\to L(\vc+\vx_1).$$
For the first case,
the HN filtration of $L(\vc)$ has the form (\ref{two term of hn for
Lc}), which yields a contradiction. For the second case, since
$\Hom(S_{1,L(\vx_1)}[-1],L)\neq 0$ and $\Hom(L,L(\vc+\vx_1))\neq 0$,
we take $W=L$ to get a contradiction to Lemma \ref{no-existence}.
This finishes the proof.
\end{proof}

Now we can say more on the final HN triangles for exceptional
sheaves.

\begin{prop}\label{x} Assume $d(\Phi)=\vx_1$ with $L, L(\vx_1)$ semistable for some line bundle $L$. Then
\begin{itemize}
  \item [(i)] $\Delta_{S_{1,L}}$ has the form $L(k\vc)\to S_{1,L}\to L(k\vc-\vx_1)[1]$, where $k=0$ or 1;
  \item [(ii)] $\Delta_{S_{1,L(\vx_1)}}$ has the form $L(\vx_1)\to
S_{1,L(\vx_1)}\to L[1]$;
  \item [(iii)] for any line bundle $L'$, if $\Delta_{L'}$ is of type
$(L(\vx), L(\vx+\vc))$, then $\vx=0$ or $-\vx_1$.
\end{itemize}
\end{prop}

\begin{proof} (i) By Proposition \ref{hn for s1i}, the final HN triangle of $S_{1,L}$ has three possibilities.
Firstly, we assume $\Delta_{S_{1,L}}$ has the form $L(k\vc)\to
S_{1,L}\to L(k\vc-\vx_1)[1]$ for some $k\in\bbZ$. We use Lemma
\ref{no-existence} to deduce contradictions for $k\neq 0,1$. More
precisely, if $k<0$, then $\Hom(L(k\vc), L)\neq 0$ and
$\Hom(L,L(k\vc-\vx_1)[1])\neq 0$, it follows that
$\varphi(L)<\varphi(L(k\vc-\vx_1)[1])$. By Lemma \ref{no-existence},
we can take $W=L$ to get a contradiction. Similarly, if $k>1$, then
we take $W=L(\vx_1)[1]$ to get a contradiction. Moreover, we can
take $W=L(\vx_1)[1]$ if $\Delta_{S_{1,L}}$ has the form
$S_{1,L}^{(2)}\to S_{1,L}\to S_{1,L(\vx_1)}[1]$, and take $W=L$ if
$\Delta_{S_{1,L}}$ has the form $S_{1,L(\vx_1)}[-1] \to S_{1,L}\to
S_{1,L(\vx_1)}^{(2)}$ to get contradictions. We are done.

The proofs for (ii) and (iii) are similar as for (1) by using Lemma
\ref{no-existence}. We omit the details.
\end{proof}

\begin{prop}\label{c} Assume $d(\Phi)=\vc$ with $L, L(\vc)$ semistable for some line bundle $L$. Then
\begin{itemize}
  \item [(i)] $\Delta_{S_{1,L}}$ has the form $L(\vc)\to S_{1,L}\to
L(\vx_1)[1]$ or $S_{1,L}^{(2)}\to S_{1,L}\to S_{1,L(\vx_1)}[1];$
  \item [(ii)] $\Delta_{S_{1,L(\vx_1)}}$ has the form $L(\vx_1)\to
S_{1,L(\vx_1)}\to L[1]$ or $S_{1,L}[-1] \to S_{1,L(\vx_1)}\to
S_{1,L}^{(2)}$;
  \item [(iii)] for any line bundle $L'$, if $\Delta_{L'}$ is of type
$(L(\vx), L(\vx+\vc))$, then $\vx=0$.
\end{itemize}
\end{prop}

\begin{proof} Assume $\Delta_{S_{1,L}}$ has the form $L(k\vc)\to S_{1,L}\to L(k\vc-\vx_1)[1]$ for some $k\in\bbZ$. If $k\leq 0$, then $\Hom(L(k\vc), L(\vc))\neq 0$ and $\Hom(L(\vc),L(k\vc-\vx_1)[1])\neq 0$, by taking $W=L(\vc)$ we get a contradiction to Lemma \ref{no-existence}. If $k>1$, then we take $W=L[1]$ to get a contradiction. Hence $k=1$ and then $\Delta_{S_{1,L}}$ has the form $L(\vc)\to S_{1,L}\to L(\vx_1)[1]$.
Now assume $\Delta_{S_{1,L}}$ has the form $S_{1,L(\vx_1)}[-1] \to
S_{1,L}\to S_{1,L(\vx_1)}^{(2)}$, then we can take $W=L$ to get a
contradiction. This finishes the proof of (i) by Proposition \ref{hn
for s1i}. Similarly, we can prove (ii) and (iii).
\end{proof}

\subsection{t-exceptional sequences}

In this subsection we will introduce the notion of a t-exceptional
sequence in a triangulated category $\cal C$, and show the existence
for $\mathcal{D}=D^b(\coh\X)$, where $\X$ is the weighted projective
line of weight type (2).

\begin{defn} Let $(\Phi,\{\Pi_{\varphi}\}_{\varphi\in\Phi})$ be a t-stability on a triangulated category $\mathcal{C}$. An exceptional sequence $E=(E_0,E_1,\cdots,E_n)$ is called a t-exceptional sequence if:

\begin{itemize}
  \item [(i)] $E_i$ is semistable, $i=0,\cdots,n$;
  \item [(ii)] $\Hom^{\leq 0}(E_i,E_j)=0$ for $0\leq i< j\leq n$;
  \item [(iii)] there exists $\varphi_0\in\Phi$ such that $\varphi(E_i)\in (\varphi_0,\tau_{\Phi}(\varphi_0)]$ for $0\leq i\leq n$.
\end{itemize}
\end{defn}

\begin{defn} A t-stability $(\Phi,\{\Pi_{\varphi}\}_{\varphi\in\Phi})$ on $\mathcal{C}$ is said to be \emph{effective} if for any two semistable objects $X, Y$, there exists $i\in\mathbb{Z}$ such that $\varphi(X[i])\leq\varphi(Y)<\varphi(X[i+1])$.
\end{defn}

Now we give the main result of this section.

\begin{thm}\label{main theorem} Let $(\Phi,\{\Pi_{\varphi}\}_{\varphi\in\Phi})$ be an effective finest t-stability on $\mathcal{D}=D^b(\coh\X)$. Then there exists a t-exceptional triple.
\end{thm}
\begin{proof} By Proposition \ref{ss-line-bundle}, we have $d(\Phi)\leq \vc$. We consider the following two cases.

(1) $d(\Phi)=\vc$. Without loss of generality, we assume that
$\co,\co(\vc)$ are semistable. Then $\co(\vx_1)$ is not semistable.
By Proposition \ref{c}, $\triangle _{\co(\vx_1)}$ has the following
two possible forms
     \begin{itemize}
     \item[(i)] $\co\to \co(\vx_1)\to S_{1,1}$;
     \item[(ii)] $S_{1,0}[-1]\to\co(\vx_1)\to \co(\vc).$
     \end{itemize}

For Case (i), we obtain that $S_{1,1}$ is semistable and
$\varphi(\co(\vc))>\varphi(\co)>\varphi(S_{1,1})$. By the
effectivity of t-stability, there exist unique $m$ and $n$ such that
$\varphi(\co(\vc)[m])$ and $\varphi(S_{1,1}[n])$ belong to the
interval $(\varphi(\co), \varphi(\co[1]))$. Clearly, $m\leq 0<n$.
Then by Remark \ref{excep} and Example \ref{kij} we
conclude that $(S_{1,1}[n],\co[1],\co(\vc)[m])$ is a t-exceptional
triple.

For Case (ii), we claim that $S_{1,0}$ is semistable and then
$\varphi(S_{1,0}[-1])>\varphi(\co(\vc))$. Indeed, if
$S_{1,0}$ is not semistable, then by Proposition \ref{c}, $\triangle_{S_{1,0}}$ has the form $S^{(2)}_{1,0}\rightarrow
S_{1,0}\rightarrow S_{1,1}[1]$. This implies that the last two triangles of the
HN-filtration of $\co(\vx_1)$ have the form
$$\xymatrix{
 S^{(2)}_{1,0}[-1] \ar[r] \ar@{.}[rd]& S_{1,0}[-1]\ar@{.}[rd]\ar[r]\ar[d]& \co(\vx_1) \ar[d]\\
  &S_{1,1}& \co(\vc).}
$$
Then $\Hom(S^{(2)}_{1,0}[-1], \co(\vc))\neq 0$ yields a
contradiction. By similar arguments as for Case (i), there exist
unique $m$ and $n$ with  $m,n\leq 0$ such that
$\varphi(\co(\vc)[m])$ and $\varphi(S_{1,0}[n])$ belong to the
interval $(\varphi(\co), \varphi(\co[1]))$. Then
$\varphi(\co(\vc)[m])<\varphi(S_{1,0}[n])$ or
$\varphi(S_{1,0}[n])<\varphi(\co(\vc)[m])<\varphi(S_{1,0}[n+1])$.
Recall that $\varphi(S_{1,0}[-1])>\varphi(\co(\vc))$. It follows
that $n+1\leq m\leq 0$ or $n+2\leq m\leq 0$ respectively. In both
cases, $(\co[1],\co(\vc)[m],S_{1,0}[n])$ is a t-exceptional triple
by Remark \ref{excep} and Example \ref{kij}.

\medskip

(2) $d(\Phi)=\vx_1$. Without loss of generality, we assume that
$\co,\co(\vx_1)$ are semistable. If $S_{1,0}$ is semistable, then
there exist unique $m$ and $n$ with $m,n\leq 0$ such that
$\varphi(\co(\vx_1)[m]),\varphi(S_{1,0}[n])\in(\varphi(\co),
\varphi(\co[1]))$. Since $\Hom(S_{1,0},\co(\vx_1)[1])\neq 0$, we
have $\varphi(S_{1,0})<\varphi(\co(\vx_1)[1])$. If
$\varphi(\co(\vx_1)[m])<\varphi(S_{1,0}[n])$, then $m\leq n$; and if
$\varphi(\co(\vx_1)[m-1])<\varphi(S_{1,0}[n])<\varphi(\co(\vx_1)[m])$,
then $m\leq n+1$. Therefore,
$$\varphi(\co)<\varphi(\co(\vx_1)[m])<\varphi(S_{1,0}[n])<\varphi(\co[1]),m\leq n\leq 0,$$
or
$$\varphi(\co)<\varphi(S_{1,0}[n])<\varphi(\co(\vx_1)[m])<\varphi(\co[1]),m\leq
n+1\leq 0.$$ We conclude that
$(\co,S_{1,0}[n-1],\co(\vx_1)[m-1])$ or
$(\co,S_{1,0}[n],\co(\vx_1)[m-1])$ is a t-exceptional triple
respectively.

Now assume $S_{1,0}$ is not semistable. Then $\triangle_{S_{1,0}}$
has the following two possible forms
     \begin{itemize}
     \item[(i)] $\co(\vc)\rightarrow S_{1,0}\rightarrow \co(\vx_1)[1]$;
     \item[(ii)] $\co\rightarrow S_{1,0}\rightarrow \co(-\vx_1)[1].$
     \end{itemize}

For Case (i), $\co(\vc)$ must be semistable with
$\varphi(\co(\vc))>\varphi(\co(\vx_1)[1])$. Otherwise,
$\triangle_{\co(\vc)}$ is given by $\co(\vx_1)\oplus \co\rightarrow
\co(\vc)\rightarrow \co(-\vx_1)[1]$. Hence, the final two triangles
of the HN-filtration of $S_{1,0}$ have the form
$$\xymatrix{
 \co(\vx_1)\oplus \co \ar[r] \ar@{.}[rd]& \co(\vc)\ar@{.}[rd]\ar[r]\ar[d]& S_{1,0} \ar[d]\\
  &\co(-\vx_1)[1]& \co(\vx_1)[1].}
$$
Then $\Hom^{-1}(\co(\vx_1),\co(\vx_1)[1])\neq 0$ yields a
contradiction. Therefore, there exist $m,n$ such that
$$\varphi(\co)<\varphi(\co(\vc)[n])<\varphi(\co(\vx_1)[m])<\varphi(\co[1]),n+2\leq m\leq 0,$$
or
$$\varphi(\co)<\varphi(\co(\vx_1)[m])<\varphi(\co(\vc)[n])<\varphi(\co[1]),n+1\leq
m\leq 0.$$ In both cases, $(\co[1],\co(\vx_1)[m],\co(\vc)[n])$ is a
t-exceptional triple.

For Case (ii), $\co(-\vx_1)$ is semistable with
$\varphi(\co)>\varphi(\co(-\vx_1)[1])$. Hence, there exist $m,n$
such that
$$(i)\varphi(\co(-\vx_1))<\varphi(\co[m])<\varphi(\co(\vx_1)[n])<\varphi(\co(-\vx_1)[1]), n\leq m\leq -1,$$
or
$$(ii)\varphi(\co(-\vx_1))<\varphi(\co(\vx_1)[n])<\varphi(\co[m])<\varphi(\co(-\vx_1)[1]),n+1\leq
m\leq -1.$$ In both cases,
$(\co(-\vx_1),\co[m],\co(\vx_1)[n-\delta_{m,n}])$ is a t-exceptional
triple.

\end{proof}

\section{Relations with Bridgeland's stability conditions}

In this section we use t-stabilities to investigate the
stability conditions in the sense of Bridgeland. By Proposition
\ref{canonical tilting}, there is a derived equivalence between the
category $\coh \X$ of coherent sheaves over the weighted projective
line $\X$ of weight type $(2)$ and the module category $\mod kQ$ for
the acyclic triangular quiver $Q$. The aim of this section is to apply the results in the previous section to derive the existence of
$\sigma$-exceptional triple for each stability condition $\sigma$ on
$D^b(\mod kQ)$, which is the main result in \cite{DK} and implies
the connectedness of the space of stability conditions on $D^b(\mod
kQ)$.

We first recall the definition of stability conditions on a
triangulated category $\mathcal{C}$ introduced by Bridgeland
\cite{TB}.

\begin{defn}\label{stab cond on tri cat} A stability condition $\sigma=(Z,\mathcal{P})$ on $\mathcal{C}$ consists of a group homomorphism
 $Z:K_0(\mathcal{C})\rightarrow\mathbb{C}$, called the \emph{central charge}, and full additive
 subcategories
$\mathcal{P}(\phi)$ of $\mathcal{C}$ for each $\phi\in\mathbb{R}$,
satisfying the following axioms:

\begin{itemize}
  \item [(i)] if $E\in\mathcal{P}(\phi)$, then
$Z(E)=m(E)\rm{exp}(i\pi\phi)$ for some $m(E)\in\mathbb{R}_{>0}$;
  \item [(ii)] for all $\phi\in\mathbb{R}$, $\mathcal{P}(\phi+1)=\mathcal{P}(\phi)[1]$;
  \item [(iii)] if $\phi _1>\phi _2$ and $A_j\in\mathcal{P}(\phi _j)$,
then $\Hom(A_1,A_2)=0$;
  \item [(iv)] for each $0\neq E\in  \mathcal{C}$, there are a finite sequence of real numbers
$\phi _1>\phi _2>\cdots>\phi _n$ and a sequence of triangles
$$\xymatrix { 0=E_0\ar[rr] &&E_1\ar[dl]\ar[rr]&&E_2\ar[dl]\ar[r]&\cdots\ar[r]&E_{n-1}\ar[rr]&&E_n=E\ar[dl]\\
&A_1\ar@{-->}[ul]&&A_2\ar@{-->}[ul]&&&&A_n\ar@{-->}[ul] }$$ with
$A_j\in\mathcal{P}(\phi _j)$ for $1\leq j\leq n$.
\end{itemize}
\end{defn}

The decompositions in axiom (iv) are uniquely defined up to
isomorphism. Given a nonzero object $0\neq E\in \mathcal{C}$, define
real numbers $\phi^{-}(E):=\phi_1$ and $\phi^{+}(E):=\phi_n$. Then
$E\in P(\phi)$ if and only if
$\phi^{-}(E)=\phi^{+}(E)=\phi=:\phi(E)$. Each subcategory
$\mathcal{P}(\phi)$ is extension-closed and abelian. Its non-zero
objects are said to be semistable of phase $\phi$ with respect to
$\sigma$, and its minimal objects are called stable.

For each interval $I\subseteq \mathbb{R}$, define
$\mathcal{P}(I):=\langle \mathcal{P}(\phi)|\phi\in I\rangle$ to be
the extension-closed subcategory of $\mathcal{C}$ generated by
$\mathcal{P}(\phi)$ for $\phi\in I$. It has been shown by Bridgeland
that for $\phi\in \mathbb{R}$, both $\mathcal{P}(\phi,\phi+1]$ and
$\mathcal{P}[\phi,\phi+1)$ are hearts of bounded t-structures on
$\mathcal{C}$.

In the following we always fix a stability condition
$\sigma=(Z,\mathcal{P})$ on $\mathcal{C}$. Now we recall the definition of
$\sigma$-exceptional sequence, which was first introduced in \cite{DK}.

\begin{defn} An exceptional sequence $\mathcal{E}=(E_0,E_1,\cdots,E_n)$ is called a  $\sigma$-exceptional sequence if

\begin{itemize}
  \item [(i)] $E_i$ is semistable, $i=0,\cdots,n$;
  \item [(ii)] $\Hom^{\leq 0}(E_i,E_j)=0,\forall\; 0\leq i< j\leq n$;
  \item [(iii)] there exists $\phi\in\mathbb{R}$ such that $\phi(E_i)\in(\phi,\phi+1]$ for each $i$.
\end{itemize}
\end{defn}

\begin{rem}
\noindent (1) The condition (iii) is equivalent to the following condition
\begin{itemize}\item [-] there exists $\phi\in\mathbb{R}$ such that $\phi(E_i)\in[\phi,\phi+1)$ for each $i$.
\end{itemize}
(2) If the condition (iii) is replaced by the condition
\begin{itemize}
  \item[-] there exists $\phi\in\mathbb{R}$ such that $\phi(E_i)\in[\phi,\phi+1]$ for each $i$,
\end{itemize}
then $\mathcal{E}$ is called a \emph{weakly $\sigma$-exceptional
sequence}.
\end{rem}

\medskip

For the later use, we give the following two general results.

\begin{lem}\label{ker and coker for ss}
Let $X\xrightarrow{f}Y\to Z$ be a triangle in $\mathcal{C}$ with
$X,Y\in\mathcal{P}(\phi)$ and $Z$ semistable. Then $\phi(Z)=\phi$ or
$\phi+1$.
\end{lem}

\begin{proof} Consider the following commutative diagram in $\mathcal{P}(\phi)$:
$$\xymatrix@C=0.5cm{
 0 \ar[rr] && \ker(f) \ar[rr] && X \ar[rr]^{f} \ar[rd]^{\pi} && Y  \ar[rr]&& \coker(f)  \ar[rr] &&0 \\
 && && & \rm{Im}(f) \ar[ru]^{i}. }$$
 By the Octahedral Axiom, we obtain the following commutative diagram of triangles:
 $$\xymatrix@C=0.5cm{
 X \ar[r]^{\pi}\ar@{=}[d] & \rm{Im}(f) \ar[r]\ar[d]^{i} & \ker(f)[1]\ar[d]^{u} \\
 X \ar[r]^{f} & Y \ar[r]\ar[d] & Z\ar[d]^{v} \\
& \coker(f) \ar@{=}[r]  & \coker(f).
 }$$
 Note that each term above is semistable and $\phi(\ker(f))=\phi(\coker(f))=\phi$. Recall that $\Hom(\mathcal{P}(\phi _1),\mathcal{P}(\phi _2))=0$ for $\phi _1>\phi _2$.
 It follows that $u=0$ or $v=0$, that is, $Z\cong \coker(f)$ or $Z\cong \ker(f)[1]$. Therefore, $\phi(Z)=\phi$ or $\phi+1$.
\end{proof}

\begin{lem}\label{mutation semistable}
Let $(E_1, E_2)$ be an exceptional pair with $\dim\Hom(E_1,
E_2[i])=1$ or 2 for some $i\in\mathbb{Z}$. If
$\phi(E_1)=\phi(E_2[i])=\phi$, then
\begin{itemize}
  \item [(i)] $\mathcal{R}_{E_2}(E_1)$ is seimstable with phase $\phi$ or $\phi+1$;
  \item [(ii)] $\mathcal{L}_{E_1}(E_2)$ is seimstable with phase $\phi-i$ or $\phi-i-1$.
\end{itemize}
\end{lem}

\begin{proof} By assumption, the triangulated subcategory $\rm{Tr}(E_1, E_2)$ is equivalent to the bounded derived category $D^b(\mod P_l)$, where $l=\dim\Hom(E_1, E_2[i])$ and $P_l$ is the $l$-Kronecker algebra, see for example \cite[Section 3.3]{EM}. In the following we show that $\mathcal{R}_{E_2}(E_1)$ is semistable. Then by Lemma \ref{ker and coker for ss},  it has phase $\phi$ or $\phi+1$, so we finish the proof of (i). The proof for the second statement is similar.

Indeed, if $l=1$, then $\rm{Tr}(E_1, E_2)$ is triangulated
equivalent to the bounded derived category of type $A_2$. Hence we
have the Auslander--Reiten triangle $E_1\to E_2[1]\to
\mathcal{R}_{E_2}(E_1)\to  E_1[1]$.
If $\mathcal{R}_{E_2}(E_1)$ is not semistable, then it has the
following HN-filtration $E_2[1]\to \mathcal{R}_{E_2}(E_1)\to
E_1[1]$. But $\phi(E_2[1])<\phi(E_1[1])$, a contradiction. Hence
$\mathcal{R}_{E_2}(E_1)$ is semistable. If $l=2$, then there is a
triangulated equivalence $\rm{Tr}(E_1, E_2)\cong\coh(\mathbb{P}^1)$.
If we view the equivalence as identity, we can assume  $(E_1,
E_2[i])=(\co, \co(1))$ without loss of generality. Then by \cite[Thm. 1.2]{SO}, all the indecomposable coherent
sheaves are semistable, in particular, $\mathcal{R}_{E_2}(E_1)$ is
semistable. We are done.

\end{proof}

Now we focus on the bounded derived category
$\cal{D}=D^{b}(\coh\mathbb{X})$ for the weighted projective line
$\mathbb{X}$ of weight type (2).

\begin{lem}\label{Kronecker embedding}
Assume $L,L(\vc)\in\mathcal{P}(\phi)$ for a line bundle $L$. Then
one of the following holds:
\begin{itemize}
  \item [(i)] there exists $n>0$ such that $\langle L(n\vc),L(n\vc+\vc)[-1]\rangle\subseteq \mathcal{P}(\phi)$;
  \item [(ii)] there exists $m>0$ such that $\langle L(-m\vc+\vc),L(-m\vc)[1]\rangle\subseteq \mathcal{P}(\phi)$.
  \end{itemize}
Consequently, $\mathcal{P}(\phi)$ contains a subcategory of the form
$\mod kP_2$, where $P_2$ is the Kronecker quiver.
\end{lem}

\begin{proof}
For any line bundle $L'$, we have the following canonical triangle
in $\mathcal{D}$: $$\xi(L'): \xymatrix@C=0.5cm{
  L'(-\vc) \ar[rr]^{(x_2,-x_1^2)^t} && L'\oplus L' \ar[rr]^{(x_1^2,x_2)} && L'(\vc) \ar[rr]&& L'(-\vc)[1] }$$
By the assumption, $L,L(\vc)\in\mathcal{P}(\phi)$. We claim that one of
the following statements holds:
\begin{itemize}
  \item[-]  for any $m>0$, $L(-m\vc)\in\mathcal{P}(\phi)$;
  \item[-]  there exists $m>0$, such that $L(-m\vc+\vc) \in\mathcal{P}(\phi)$ and $L(-m\vc)[1] \in\mathcal{P}(\phi)$.
   \end{itemize}
   In fact, from the canonical triangle $\xi(L)$ and by Lemma \ref{mutation semistable} we get $L(-\vc)$ is semistable with phase $\phi$ or $\phi-1$. If $\phi(L(-\vc))=\phi-1$, then the second statement holds by taking $m=1$. If $\phi(L(-\vc))=\phi$, by considering the canonical triangle $\xi(L(-\vc))$ we get $L(-2\vc)$ is semistable with phase $\phi$ or $\phi-1$. Repeating the procedure gives the proof of the claim.

Dually, one of the following statements holds:
\begin{itemize}
  \item[-]  for any $n>0$, $L(n\vc)\in\mathcal{P}(\phi)$;
  \item[-]  there exists $n>0$, such that $L(n\vc) \in\mathcal{P}(\phi)$ and $L(n\vc+\vc)[-1] \in\mathcal{P}(\phi)$.
   \end{itemize}

Therefore, we conclude that one of the following holds:
\begin{itemize}
  \item [(i)] for any $m\in\mathbb{Z}$, $L(m\vc)\in\mathcal{P}(\phi)$; in this case, we get a contradiction to the fact that $\mathcal{P}(\phi)$ is of finite length;
  \item [(ii)] there exists $m>0$, such that $L(-m\vc+\vc) \in\mathcal{P}(\phi)$ and $L(-m\vc)[1] \in\mathcal{P}(\phi)$; in this case, $\langle L(-m\vc+\vc),L(-m\vc)[1]\rangle\subseteq \mathcal{P}(\phi)$;
  \item [(iii)] there exists $n>0$, such that $L(n\vc) \in\mathcal{P}(\phi)$ and $L(n\vc+\vc)[-1] \in\mathcal{P}(\phi)$; in this case, $\langle L(n\vc),L(n\vc+\vc)[-1]\rangle\subseteq \mathcal{P}(\phi)$.
  \end{itemize}
\end{proof}

In the following we show the existence of a $\sigma$-exceptional
triple in $\mathcal{D}$ for certain special cases.

\begin{lem}\label{dim two case} Assume $L,L(\vc)\in\mathcal{P}(\phi)$ for a line bundle $L$. If $S_{1,L}[-i]$ or $\tau S_{1,L}[i-1]\in\mathcal{P}(\phi)$ for some $i>0$, then there exists a $\sigma$-exceptional triple.
\end{lem}

\begin{proof} Without loss of generality, we assume $L=\co$ and prove the result under the assumption that $S_{1,0}[-i]\in\mathcal{P}(\phi)$. The proof under the assumption that $\tau S_{1,L}[-i-1]\in\mathcal{P}(\phi)$ can be treated similarly. By Lemma \ref{Kronecker embedding}, there exists some $n>0$, such that $(\co(-n\vc)[1],\co(-n\vc+\vc))$ or $(\co(n\vc),\co(n\vc+\vc)[-1])$ is an exceptional pair in $\mathcal{P}(\phi)$.
For the first case, we get
$(\co(-n\vc)[1],\co(-n\vc+\vc),S_{1,0}[-i])$ is a
$\sigma$-exceptional triple. For the second case, consider the
exceptional triple
$\mathcal{E}:=(\co(n\vc),\co(n\vc+\vc)[-1],S_{1,0}[-i])$ in
$\mathcal{P}(\phi)$. If $i>1$, then $\mathcal{E}$ is a
$\sigma$-exceptional triple.  If $i=1$, then from the triangle
$$\co(n\vc+\vx_1)[-1]\to\co(n\vc+\vc)[-1]\to S_{1,0}[-1]$$ and by
Lemma \ref{mutation semistable} we know that
$\phi(\co(n\vc+\vx_1))=\phi+1$ or $\phi$.
\begin{itemize}
  \item[-] If $\phi(\co(n\vc+\vx_1))=\phi+1$, then $(\co(n\vc),S_{1,0}[-1],\co(n\vc+\vx_1)[-1])$ is a $\sigma$-exceptional triple.
  \item[-] If $\phi(\co(n\vc+\vx_1))=\phi$, then from the triangle $$\co(n\vc)\to\co(n\vc+\vx_1)\to S_{1,1}$$ we know that $\phi(S_{1,1})=\phi+1$ or $\phi$. It follows that $(\co(n\vc+\vx_1),S_{1,1}[-1],\co(n\vc+\vc)[-1])$ or $(S_{1,1},\co(n\vc),\co(n\vc+\vc)[-1])$ is a $\sigma$-exceptional triple.
\end{itemize}
\end{proof}

An exceptional triple $(E_0,E_1,E_2)$ is said to be of $\sigma$-type
$(\phi_1,\phi_2,\phi_3)$ if $E_i$ is $\sigma$-semistable and
$\phi(E_i)=\phi_i$ for $1\leq i\leq 3$.

\begin{lem}\label{weak-exc-seq-eq} Let $(E_0,E_1,E_2)$ be an Ext-exceptional triple in $\mathcal{D}$ of $\sigma$-type $(\phi+1,\phi,\phi)$ or $(\phi+1,\phi+1,\phi)$.
Then there is a $\sigma$-exceptional triple in $\mathcal{D}$.
\end{lem}
\begin{proof}
We only prove the existence of $\sigma$-exceptional triple if
$(E_0,E_1,E_2)$ has $\sigma$-type $(\phi+1,\phi,\phi)$. The proof
for the other one is similar. By Proposition \ref{exceptional pair
and triple} $\dim\Ext^1(E_0,E_1)\leq 2$. If $\Ext^1(E_0,E_1)=0$,
then by Lemma \ref{12to02} we have $\Ext^1(E_0,E_2)=0$, thus
$(E_0[-1],E_1,E_2)$ is a $\sigma$-exceptional triple. If
$\dim\Ext^1(E_0,E_1)=2$, then we can assume
$(E_0,E_1)=(\co[1],\co(\vc))$ without loss of generality. It follows
that $E_2=S_{1,0}[-i]$ for some $i>0$, then we are done by Lemma
\ref{dim two case}. Hence we assume $\dim\Ext^1(E_0,E_1)=1$ in the
following. Denote by $W_1=\mathcal{L}_{E_0}(E_1)$, then there is a
triangle $W_1\to E_0[-1]\to E_1\to W_1[1]$. It follows that
$\phi(W_1)=\phi-1$ or $\phi$.

Assume $\phi(W_1)=\phi-1$, then
$\mathcal{E}_1:=(W_1[1],E_0[-1],E_2)$ is an exceptional triple in
$\mathcal{P}(\phi)$. By Remark \ref{excep} and Example
\ref{kij} we know that $\dim\Hom(E_0[-1],E_2)\leq 1$. If
$\dim\Hom(E_0[-1],E_2)=0$, then $\mathcal{E}_1$ is a
$\sigma$-exceptional triple. If $\dim\Hom(E_0[-1],E_2)=1$, then
$W_2:=\mathcal{L}_{E_0}(E_2)$ is semistable which fits into the
triangle $W_2\to E_0[-1]\to E_2\to W_2[1]$. Hence $\phi(W_2)=\phi$
or $\phi-1$.  It follows that $(W_1[1],W_2[1],E_0[-1])$ or
$(W_1[1],E_2,W_2)$ is a $\sigma$-exceptional triple,
since $\Hom(W_1[1],W_2[1])\cong\Hom(E_1,W_2[1])=0$ and
$\Hom(W_1[1],E_2)=0$.

Now assume $\phi(W_1)=\phi$, then $\mathcal{E}_2:=(E_1,W_1,E_2)$ is
an exceptional triple in $\mathcal{P}(\phi)$ with
$\dim\Hom(W_1,E_2)\leq 2$. If $\dim\Hom(W_1,E_2)=0$, then
$\mathcal{E}_2$ is a $\sigma$-exceptional triple. If
$\dim\Hom(W_1,E_2)=1$, then $W_3:=\mathcal{L}_{W_1}(E_2)$ is
semistable which fits into the triangle  $W_3\to W_1\to E_2\to
W_3[1]$. Hence $\phi(W_3)=\phi$ or $\phi-1$. It follows that
$(E_1,E_2,W_3)$ or $(E_1,W_3[1],W_1)$ is a $\sigma$-exceptional
triple. If $\dim\Hom(W_1,E_2)=2$, then we can assume
$(W_1,E_2)=(\co,\co(\vc))$ without loss of generality. It follows
that $E_1=S_{1,1}[i-1]$ for some $i>0$, then we are done by Lemma
\ref{dim two case}. This finishes the proof.
\end{proof}

\begin{prop}\label{closed interval to open interval}
If $\mathcal{D}$ admits a weakly $\sigma$-exceptional triple, then
there is a $\sigma$-exceptional triple in $\mathcal{D}$.
\end{prop}

\begin{proof}
Let $\mathcal{E}:=(E_0,E_1,E_2)$ be a weakly $\sigma$-exceptional
triple in $\mathcal{D}$ with $\phi(E_i)\in[\phi,\phi+1]$ for some
$\phi$. If $\phi(E_0)=\phi$ or $\phi(E_2)=\phi+1$, then
$(E_0[1],E_1,E_2)$ or $(E_0,E_1,E_2[-1])$ is a weakly
$\sigma$-exceptional triple respectively. Therefore, it suffices to
show the existence of a $\sigma$-exceptional triple under the
assumptions $\phi(E_0)>\phi$ and $\phi(E_2)<\phi+1$. Obviously, if
$\phi(E_i)\in[\phi,\phi+1)$ or $(\phi,\phi+1]$ for each $0\leq i\leq
2$, then by definition,  $\mathcal{E}$ is a $\sigma$-exceptional
triple, we are done. Now assume that there exists a permutation
$s\in S_3$ such that $\phi(E_{s(0)})=\phi$ and
$\phi(E_{s(1)})=\phi+1$. Now if $\phi(E_{s(2)})=\phi$ or $\phi+1$,
then by Lemma \ref{weak-exc-seq-eq}, we are done. Hence, we can
further assume that $\phi(E_{s(2)}):=\psi\in(\phi,\phi+1)$.
Therefore, $\mathcal{E}$ has one of the following $\sigma$-type:
$$(i)\,(\phi+1,\phi,\psi);\; \quad(ii)\,(\psi, \phi+1,\phi);\; \quad(iii)\,(\phi+1,\psi,\phi).$$

For the first case, we get that
$\mathcal{E'}:=(E_0,E_1[1],E_2)\subseteq\mathcal{P}(\phi,\phi+1]$.
If $\Ext^1(E_0,E_1)=0$, then by definition $\mathcal{E'}$ is a
$\sigma$-exceptional triple. If $\Ext^1(E_0,E_1)\neq 0$, then
$\dim\Ext^1(E_0,E_1)=1$ or 2 by Proposition \ref{exceptional pair
and triple}.

\emph{Case 1:} $\dim\Ext^1(E_0,E_1)=1$; then by Lemma \ref{mutation
semistable}, $\mathcal{R}_{E_1}(E_0)$ is seimstable with phase
$\phi+1$ or $\phi+2$. It follows that
$(\mathcal{R}_{E_1}(E_0),E_0,E_2)$ or
$(E_1[1],\mathcal{R}_{E_1}(E_0)[-1],E_2)$ is a $\sigma$-exceptional
triple.

\emph{Case 2:} $\dim\Ext^1(E_0,E_1)=2$; then we can assume
$(E_0,E_1)=(\co[1],\co(\vc))$ without loss of generality. It follows
that $E_2=S_{1,0}[-i]$ for some $i>0$ and
$\co,\co(\vc)\in\mathcal{P}(\phi)$. Then from Lemma \ref{Kronecker
embedding}, there exists $n>0$ such that $\langle
\co(n\vc-\vc),\co(n\vc)[-1]\rangle\subseteq \mathcal{P}(\phi)$, or
there exists $m>0$ such that $\langle
\co(-m\vc+\vc),\co(-m\vc)[1]\rangle\subseteq \mathcal{P}(\phi)$. It
follows that $(\co(n\vc-\vc),\co(n\vc)[-1],E_2)$ or
$(\co(-m\vc+\vc),\co(-m\vc)[1],E_2)$ is a $\sigma$-exceptional
triple.

For the second case,
$\mathcal{E'}:=(E_0,E_1[-1],E_2)\subseteq\mathcal{P}[\phi,\phi+1)$,
the proof is dual. For the third case, if $\Hom(E_1,E_2[1])=0$, then
by Lemma \ref{12to02} we get $\Hom(E_0,E_2[1])=0$. It follows that
$\mathcal{E'}$ is a $\sigma$-exceptional triple. Now assume that
$\Hom(E_1,E_2[1])\neq0$. Note that
$(E_1,E_2)\subseteq\mathcal{P}[\phi,\phi+1)$. We obtain from
\cite[Prop.3.15]{DK} that $\sigma$ induces a stability
condition $\sigma'=(Z',\mathcal{P}')$ on $\rm{Tr}(E_1,E_2)$
satisfying $\mathcal{P}'(t)=\mathcal{P}(t)\bigcap\rm{Tr}(E_1,E_2)$
for any $t\in\mathbb{R}$. Then by \cite[Prop.3.17]{EM}, the right mutation $\mathcal{R}_{E_2}(E_1)$ is semistable with phase $\psi'$ in the interval $(\phi+1,\phi+2)$. Hence we get
an exceptional triple $(E_0,E_2[1],\mathcal{R}_{E_2}(E_1)[-1])$,
which has $\sigma$-type $(\phi+1,\phi+1,\psi'-1)$, then by an
argument as for the first case we can finish the proof.
\end{proof}

Now we can show that Theorem \ref{main theorem} implies the main
theorem in \cite{DK} which is the key to prove that the space of
stability conditions on $\mathcal{D}$ is connected.

\begin{thm}{\rm(}\cite[Thm. 10.1]{DK}{\rm)} Let $\sigma=(Z,\mathcal{P})$ be a stability condition on $\mathcal{D}$. Then there exists a $\sigma$-exceptional triple.
\end{thm}
\begin{proof} By the definition, $(\mathbb{R},\{\mathcal{P}(\phi)\}_{\phi\in\mathbb{R}})$ is a t-stability on $\mathcal{D}$, which can be refined to a finest one, say, $(\Phi,\{\Pi_{\varphi}\}_{\varphi\in\Phi})$. Then there exists a surjective map $r:\Phi\rightarrow \mathbb{R}$ such that for each $\varphi\in\Phi$, $r(\tau_{\Phi}(\varphi))=\tau_{\mathbb{R}}(r(\varphi))=r(\varphi)+1$.  Hence, for any $\varphi,\varphi'\in\Phi$, there exists $i\in\mathbb{Z}$ such that
$$r(\tau_{\Phi}^{i}(\varphi))=r(\varphi)+i\leq r(\varphi')<r(\varphi)+i+1=r(\tau_{\Phi}^{i+1}(\varphi)).$$
It follows that
$\tau^{i}_{\Phi}(\varphi)\leq\varphi'<\tau^{i+1}_{\Phi}(\varphi)$.
Therefore, $(\Phi,\{\Pi_{\varphi}\}_{\varphi\in\Phi})$ is an
effective finest t-stability. By Theorem \ref{main theorem}, there
exists a t-exceptional triple $(E_0,E_1,E_2)$ in $\mathcal{D}$.
Hence, there exists some $\varphi_0\in\Phi$ such that
$\varphi(E_i)\in (\varphi_0,\tau_{\Phi}(\varphi_0)]$ for each $i$.
It follows that $E_i\in\mathcal{P}[r(\varphi_0),r(\varphi_0)+1]$ for
each $i$. Then by Proposition \ref{closed interval to open
interval}, there exists a $\sigma$-exceptional triple on
$\mathcal{D}$.
\end{proof}

\section*{Acknowledgements}

The authors would like to thank Bangming Deng and Zhe Han for their
valuable suggestions. We are also grateful to Nan Gao
for providing the proof of Lemma \ref{kernel and cokernel}.

\end{document}